\documentclass[12pt]{amsart}

\usepackage{verbatim}
\usepackage{hyperref}
\usepackage{cite}
\usepackage{adjustbox}
\usepackage{enumitem}
\usepackage{amssymb}
\usepackage{amsmath}
\usepackage{amsthm}
\usepackage{amsfonts}
\usepackage{float}
\usepackage{graphicx}%
\usepackage{cancel}
\usepackage{comment}
\usepackage{bbm}
\usepackage{xcolor}
\usepackage[toc,page]{appendix}
\usepackage{mathtools}
\usepackage{soul}
\newcommand{\Exp}{\operatorname{Exp}}
\newcommand{\Cut}{\operatorname{Cut}}

\newcommand{\dd}{\,\mathrm{d}}

\usepackage[plain]{fullpage}
\usepackage[stretch=10]{microtype}
\usepackage{todonotes}
\usepackage{bm}
\usepackage{enumitem}

\numberwithin{equation}{section}
\theoremstyle{plain}
\newtheorem{theorem}{Theorem}[section]
\newtheorem{proposition}[theorem]{Proposition}
\newtheorem{lemma}[theorem]{Lemma}
\newtheorem{corollary}[theorem]{Corollary}

\theoremstyle{definition}
\newtheorem{definition}[theorem]{Definition}
\newtheorem*{defs*}{Definition}

\theoremstyle{remark}
\newtheorem{remark}[theorem]{Remark}
\newtheorem*{ass*}{Assumption}
\newtheorem*{ack*}{Acknowledgements}

\newcommand{\R}{\mathbb{R}}

\usepackage{setspace}

\title{Optimal Synthesis on a Radially Symmetric Grushin Space}

\author{Michael Albert}

\address{
Okinawa Institute of Science and Technology (OIST),
Onna, Okinawa 904-0495, Japan
}

\email{michael.albert@oist.jp}



\thanks{
\textbf{Acknowledgements:} This work was supported in part by the National Science Foundation
under Grant Number DMS-2246549. Special thanks to Masha Gordina and Samu\"{e}l Borza for their input on the manuscript
}

\date{\today}

\begin{document}
\begin{abstract}
We study a class of radially symmetric Grushin-type structures on \(\mathbb{R}^3\) generated by $X=\partial_x, Y=\partial_y, Z_f=f(r)\partial_z$, with $r$ the cylindrical radius, and \(f\) positive away from the vertical axis, increasing, and unbounded. Under a natural monotonicity assumption on the elasticity \(\eta(r)=r f'(r)/f(r)\), we obtain the optimal synthesis from singular points. For profiles satisfying an additional monotonicity constraint on the radial effective potential which ensures that geodesics are either radially periodic, circular or planar straight lines, we demonstrate a continuous, globally defined cut time candidate for Riemannian geodesics, which coincides with the radial period, evidently a Maxwell time, for all radially periodic geodesics.  

We then specialize to the homogeneous family \(f(r)=r^\alpha\), \(\alpha>1\), and determine the full optimal synthesis from Riemannian points. The cut time is shown to coincide with the complete radial period. Finally, we identify the resulting cut locus, which forms a pair of properly embedded two-dimensional regions on the cylinder through the initial point.  

Finally, we obtain a necessary condition for the constructed Maxwell time to be the true cut time, and rule it out as the global cut mechanism for the full class under consideration. 
\end{abstract}
\maketitle
\tableofcontents

\section{Introduction}
We consider the following vector fields on $\mathbb{R}^3$. Let
\begin{align} X=\partial_x,\qquad Y=\partial_y,\qquad Z_f=f(r)\partial_z,\qquad q=(x,y,z)\in \mathbb{R}^3,\end{align}
where $r=\sqrt{x^2+y^2}$ is the cylindrical radius in $\mathbb{R}^3$. We define the family $\mathfrak{F}$ by the following properties. We say that $f\in \mathfrak{F}$ if
\begin{enumerate}[label=\textnormal{(A\arabic*)}]
    \item\label{A1} $f\in C([0,\infty))\cap C^2((0,\infty))$, with
$f(0)=0$ and $f(r)>0$ for every $r>0$.
    \vspace{.25cm}
    \item\label{A2} $A(r):=\frac{f(r)f'(r)}{r}$ extends to be continuous at $r=0$.
    \vspace{.25cm} 
    \item\label{A3} $f(r)\to+\infty$ as $r\to+\infty$.
    \vspace{.25cm}
    \item \label{A4} The \emph{elasticity} of $f$, denoted $\eta(r):=\frac{rf'(r)}{f(r)}$, is a nonincreasing function of $r$.

\end{enumerate}
These assumptions imply that $\eta(r)>0$ for every $r>0$, hence that $f$ is strictly increasing. Indeed, if $\eta(r_0)\le0$ for some $r_0>0$, then by monotonicity
$\eta(r)\le0$ for every $r\ge r_0$, so $f$ would be nonincreasing
on $[r_0,\infty)$, contradicting \ref{A3}. Put 
\[
    \Phi(\xi)
    :=
    f\!\left(\lvert \xi\rvert\right)^2,\qquad \xi\in \mathbb{R}^2.
\]
We additionally say that $f\in \mathfrak{F}_R$ if it holds that $f\in \mathfrak{F}$ and $\Phi\in C^2(\mathbb{R}^2)$. Observe then that this is equivalent to $rA'(r)\to 0$ as $r\downarrow 0$. Introduce also the function 
\[\Psi(r)=r^3f(r)f'(r),\qquad r>0.\]
We say that $f\in \mathfrak{F}_{R}^{\mathbf{P}}$ if $f\in \mathfrak{F}_R$ and the function $\Psi$ is strictly increasing.  

Our objective in this paper is to study the Carnot-Carth\'{e}odory (CC) geometry generated by $\{X,Y,Z_f\}$ under the conditions described above. Our focus will be the classical problem of optimal synthesis. That is, for a given point $q_0\in \mathbb{R}^3$, determine all geodesics starting at $q_0$, and as a function of the geodesic parametrized by its initial covector, write down the cut time, and geometrically characterize the set of points in the \emph{cut locus}.

This is a Grushin type space, which is Riemannian away from the \emph{singular set} $\Sigma:=\{r=0\}$. Canonical examples of the function $f\in \mathfrak{F}$ include the \emph{power law functions} $f(r)=r^\alpha$ for $\alpha\geq 1$. As such, the resulting length space that we consider has similar geometric properties to the $\alpha$-Grushin plane, higher dimensional $\alpha$-Grushin spaces, and related constructions, as discussed in a variety of sources. A non-exhaustive list of sources that have considered Grushin style spaces is \cite{Borza2022,AgrachevBarilariBoscainBook2020,BoscainNeel2020,WuJang-Mei2015,BieskeGong2006,TurcanuUdriste2017,ChangLi2012,GalloneMichelangeliPozzoli2019,albert2025geodesicsgrushinspaces,BAUER2015188,abatangelo2024solutionsclassdegenerateequations,LIU2018237}. Our setting is most similar to that discussed in \cite{abatangelo2024solutionsclassdegenerateequations,BAUER2015188,LIU2018237}. 

The problem of determining cut loci and optimal syntheses in related sub-Riemannian settings has also been studied extensively; see, for example, \cite{AgrachevBoscainSigalotti2008,BarilariBoscainNeel2012,Charlot2014,Sachkov2010,Sachkov2011,ArdentovSachkov2015,BonnardCaillauSinclairTanaka2009,BonnardCaillauJanin2013}. Of particular relevance is the work of Charlot \cite{Charlot2014}, which gives a detailed local description of the cut locus and optimal synthesis near a classical step-two Grushin point in a 2 dimensional almost Riemannian manifold (See \cite[Ch. 9]{AgrachevBarilariBoscainBook2020}). The present class allows substantially more degenerate singular behavior, and therefore lies outside that local step-two framework.

The hypotheses defining $\mathfrak F$ are tailored to the synthesis from
singular points.  Conditions \ref{A1}--\ref{A3} ensure that the singular
set is exactly the vertical axis, provide the regularity needed for the
Hamiltonian formulation of geodesics, and guarantee that nontrivial
trajectories issued from the singular axis have a unique radial turning
point.  The monotonicity assumption \ref{A4} implies an assumption that was imposed directly in a previous version of the paper, namely 
\[
    \frac{f(r)^2}{f'(r)}\longrightarrow+\infty
    \qquad\text{as }r\to+\infty,
\]
but is stronger. Its key role in the analysis is
forcing the radial period function
$\mathbf P(E,0,w_0)$ in \eqref{Period Formula} to be strictly decreasing
with $|w_0|$.  This monotonicity orders the family of trajectories
issued from the singular axis and, together with the preceding growth
property, gives the shooting argument used to prove the singular
optimal synthesis, following the strategy of
\cite{AgrachevBarilariBoscainBook2020,Borza2022}.

The subclasses $\mathfrak F_R$ and $\mathfrak F_R^{\mathbf P}$ are
introduced for the analysis at Riemannian points. The condition
$f\in\mathfrak F_R$ is precisely the additional regularity needed so the associated Hamiltonian vector field of the structure
is $C^1$, and hence its flow and the exponential map depend $C^1$ on the
initial covector even for trajectories which cross the axis. In our model this will occur for the level set of covectors which have zero angular momentum. The further
condition defining $\mathfrak F_R^{\mathbf P}$, namely the strict
monotonicity of $\Psi(r)=r^3f(r)f'(r)$, gives the radial effective
potential a unique critical radius. In this class, the radial period is
therefore a natural globally defined candidate for the Riemannian cut
mechanism. In particular for $\alpha>1$, the power laws $r\mapsto r^{\alpha}$ are in $\mathfrak F_R^{\mathbf P}$. The classes $\mathfrak F$, $\mathfrak F_R$, and
$\mathfrak F_R^{\mathbf P}$ contain many profiles which are not
monomials.  Some representative examples, including profiles of future interest
are discussed in Section~7.

A departure from classical sub-Riemannian theory is that $\mathfrak{F}$ allows
$f$ to vanish to all orders on $\Sigma$.  For such profiles,
H\"ormander's condition fails at every point of the singular set. Because H\"ormander's condition may fail on $\Sigma$, the
Chow--Rashevskii theorem is not directly applicable.  Nevertheless,
every two points can be joined by an admissible curve, and in
Theorem~\ref{metric} we show that the resulting Carnot--Carath\'eodory
distance $d_{CC}$ is a complete metric inducing the Euclidean topology
on $\mathbb R^3$.

Our first main result is the complete optimal synthesis from singular
points, proved in Theorem~\ref{singular point optimal synth}.  For every
$f\in\mathfrak F$, the cut time of any non-straight line normal geodesic
issued from $\Sigma$ is its first return time to $\Sigma$, equal to the (physical) radial period, and the cut
locus of a singular point is the singular axis with the initial point
removed.  This is the same qualitative picture as in the classical
Grushin plane.

We then turn to Riemannian initial points. Proposition~\ref{prop:radial-periodicity}
ensures that for $f\in \mathfrak{F}_R^{\mathbf{P}}$, normal geodesics are radially periodic, and
Theorem~\ref{Geodesic and Cut time Theorem} produces
Maxwell upper bounds for the cut time.  In the homogeneous family
$f(r)=r^\alpha$, $\alpha>1$, we show by the Extended Hadamard technique,
Theorem~\ref{extended Hadamard}, that the radial-period Maxwell time is
the true cut time and obtain the complete Riemannian synthesis in
Theorem~\ref{thm:riemannian-synthesis}.  The resulting cut locus,
illustrated in Figure~\ref{cut locus fig}, is a union of tongue-shaped
properly embedded half-planes contained in the cylinder
$\{r=r_0\}$.  In particular it has codimension one.  This contrasts
with the case $\alpha=1$, where the Riemannian cut locus has
codimension two \cite{albert2026measurecontractionpropertygrushin}.

Outside the homogeneous family, the Riemannian optimal synthesis is
more delicate.  The conjugate-time analysis depends on finer
properties of the profile, and the radial period is not universally
the cut time even within $\mathfrak F_R^{\mathbf P}$.  Section~7 contains
an explicit profile for which the radial-period Maxwell time fails to
be sharp and in
Theorem~\ref{thm:necessary-radial-period} we give a necessary condition for
$\mathbf P$ to equal the cut time for every Riemannian geodesic.
We conclude with numerical evidence showing that this necessary
condition alone is unlikely to be sufficient.
\section{Preliminaries}
Due to a lack of H\"{o}rmander's condition on $\Sigma$, the vector fields $\{X,Y,Z_f\}$ do not induce a sub-Riemannian structure on $\mathbb{R}^3$ in the usual sense, e.g as in \cite{AgrachevBarilariBoscainBook2020}. We will call a metric structure on a smooth manifold $M$ arising from the CC-distance induced by a family of H\"{o}rmander vector fields a \emph{pure sub-Riemannian structure}.

Many of the usual tools deployed in the study of sub-Riemannian structures will still be applicable to our setting, and we will find that not much is lost due to the lack of regularity and H\"{o}rmander's condition. We proceed by defining the notions of \emph{admissible curve}, \emph{metric length} and \emph{the Carnot-Carth\'{e}odory distance.} 
\begin{definition}
    A $W^{1,2}([0,T],\mathbb{R}^3)$ curve $\gamma=(x,y,z):[0,T]\rightarrow \mathbb{R}^3$ is called \emph{admissible} if there is a \emph{control} $u=(u_1,u_2,u_3)\in L^2([0,T];\mathbb{R}^3)$ such that \begin{align}\label{control definition} \gamma'(t)=u_1(t)X(\gamma(t))+u_2(t)Y(\gamma(t))+u_3(t)Z_f(\gamma(t)),\qquad \text{a.e}\,\, t\in [0,T].\end{align} 
\end{definition}
\begin{definition}
    The \emph{metric length} of an admissible curve $\gamma$ is
    \[
        \ell(\gamma)
        :=
        \int_0^T |u^*(t)|\,dt,
    \]
    where $u^*(t)$ is the unique control of minimal Euclidean norm
    representing $\gamma'(t)$. The curve $\gamma$ is \emph{parametrized by constant speed} if
    $|u^*(t)|=c$ for some $c>0$ almost everywhere.  It is
    \emph{arc-length parametrized}, or \emph{unit speed}, if $c=1$.
\end{definition}
Notice that admissibility in the above definition forces $\dot z(t)=0$ almost everywhere on $\{r=0\}$. The proof that the minimal control $u^*$ is measurable and can be taken to be in $L^2$ is non-trivial, but applies to our setting with minimal surgery. See for example, \cite[Chapter 3.1]{AgrachevSachkov}.
\begin{definition}
    Let $T>0$. Let $\gamma=(x,y,z):[0,T]\rightarrow \mathbb{R}^3$ be an admissible curve. Put $q_1=\gamma(0)$ and $q_2=\gamma(T)$. Another admissible curve $\tilde{\gamma}:[0,\tilde{T}]\rightarrow \mathbb{R}^3$ with the same endpoints as $\gamma$ is called a \emph{competitor} to $\gamma$. If $\ell(\gamma)\leq \ell(\tilde{\gamma})$ for all competitors $\tilde{\gamma}$, then $\gamma$ is called a \emph{length minimizer}. If for every $t\in [0,T)$, there is $\varepsilon>0$ such that $\gamma\rvert_{[t,t+\varepsilon]}$ is a length minimizer parametrized by constant speed, then $\gamma$ is called a \emph{geodesic}.
\end{definition}
See Chapter 5.1 of \cite{HeinonenKoskelaShanmugalingamTyson2001} for a proof applicable to our setting that all finite length admissible curves can be reparametrized without affecting the value of the metric length $\ell(\gamma)$.
\begin{theorem}\label{metric}
Let $f\in \mathfrak{F}$. The Carnot-Carth\'{e}odory metric 
    \begin{align}
        d_{CC}(q,q')=\inf\{\ell(\gamma): \gamma\,\,\text{is admissible}\,\,,\gamma(0)=q,\gamma(T)=q'\}
    \end{align}
    is well defined and induces the Euclidean topology on $\mathbb{R}^3.$ Furthermore, $(\mathbb{R}^3,d_{CC})$ is a complete metric space.
\end{theorem} 
\begin{proof}
By concatenating planar straight lines with off-axis vertical straight lines, one may see that $d_{CC}(\cdot,\cdot)$ is well defined and finite. We then show that the topology generated by $d_{CC}$ agrees with the Euclidean topology. Let $q_0=(x_0,y_0,z_0)$, write $r_0=\sqrt{x_0^2+y_0^2}$, and let $\gamma:[0,T]\to\mathbb R^3$ be an arclength parametrized admissible curve beginning at $q_0$, with controls $u_1,u_2,u_3$ satisfying $u_1^2+u_2^2+u_3^2=1$ almost everywhere. Since $|\dot r|\leq \sqrt{u_1^2+u_2^2}\leq1$, we have $r(t)\leq r_0+T$. Moreover,

$$
|(x(T),y(T))-(x_0,y_0)|\leq T,\qquad
|z(T)-z_0|\leq \int_0^T |u_3(t)|f(r(t))\,dt\leq T f(r_0+T).
$$
Pass to the infimum over admissible curves in the above inequalities. Then, take $d_{CC}(q_0,q)\to 0$, which by the continuity of $f$, then implies $q\to q_0$ in the Euclidean topology. For the converse, first suppose that $r_0>0$. If $q=(x,y,z)$ is sufficiently close to $q_0$, the planar segment joining $(x_0,y_0,z_0)$ to $(x,y,z_0)$ remains in a compact subset of $\{r>0\}$. Moving vertically at $(x,y,z_0)$, we obtain

$$
d_{CC}(q_0,q)\leq |(x,y)-(x_0,y_0)|+\frac{|z-z_0|}{f(r)},
$$
which tends to zero as $q\to q_0$. Suppose now that $q_0\in\Sigma$, and by translation take $q_0=(0,0,0)$. Let $r=\sqrt{x^2+y^2}$ and put $\rho=h(|z|)$, where $h$ is the inverse of the strictly increasing function $s\mapsto sf(s)$. If $z\neq0$, move horizontally from the origin to a point of radius $\rho$, move vertically there by $z$, and then move horizontally to $(x,y,z)$. Choosing the intermediate point on the same radial ray as $(x,y)$ when $r\neq0$, this path has length at most

$$
\rho+\frac{|z|}{f(\rho)}+|\rho-r|
=2\rho+|\rho-r|\leq3h(|z|)+r.
$$

The same estimate is immediate when $z=0$. Since $h(t)\to0$ as $t\downarrow0$, Euclidean convergence to a point of $\Sigma$ therefore implies convergence in $d_{CC}$. Thus the two topologies agree.

It remains to prove completeness. To that end, we show that closed $CC$-balls are compact. Let $q\in\overline B_{CC}(q_0,R)$. For every $\varepsilon>0$ there is an admissible curve from $q_0$ to $q$ of length at most $R+\varepsilon$. The preceding estimates give

$$
r_q\leq r_0+R+\varepsilon,\qquad
|z_q-z_0|\leq (R+\varepsilon)f(r_0+R+\varepsilon).
$$

Hence every closed $CC$-ball is Euclidean bounded. Since the Euclidean and $CC$ topologies coincide, a closed $CC$-ball is also Euclidean closed and is therefore compact. Thus $(\mathbb R^3,d_{CC})$ is complete.
\end{proof}
We now seek a greater understanding of length minimizers. Indeed, although it must be justified, the length minimizers in $(\mathbb{R}^3,d_{CC})$ can still be understood at the level of controls through Hamiltonian theory as is the case for purely sub-Riemannian structures in the sense of \cite{AgrachevBarilariBoscainBook2020}. 
\subsection{Hamiltonian formulation}
We next associate a Hamiltonian system to the control problem described in the previous section. Again assume that $f\in \mathfrak{F}$. In
canonical coordinates on the cotangent bundle, put 
\[
    (q,\lambda)=(x,y,z,u,v,w)
    \in T^*\R^3\simeq\R^3\times\R^3,
\]
and define the \emph{Hamiltonian function} associated to $f$ by
\begin{equation}\label{eq:Hamiltonian}
    H_f(q,\lambda)
    :=
    \frac12\left(u^2+v^2+f(r)^2w^2\right),
    \qquad
    r=\sqrt{x^2+y^2}.
\end{equation}
Since $\partial_x\bigl(f(r)^2\bigr)=2A(r)x$ and $\partial_y\bigl(f(r)^2\bigr)=2A(r)y$ away from the axis, with both expressions extending continuously by
zero at $r=0$, the Hamiltonian $H_f$ belongs to
$C^1(T^*\R^3)$ and is $C^2$ away from $T^*\Sigma$.

Let $\vartheta=u\,dx+v\,dy+w\,dz$ 
be the Liouville one-form and let $\sigma=-d\vartheta$ be the
canonical symplectic form on $T^*\R^3$.  The Hamiltonian vector field
$\vec H_f$ is defined by
\[
    \sigma(\,\cdot\,,\vec H_f)=dH_f.
\]
It is continuous on $T^*\R^3$ and $C^1$ away from $T^*\Sigma$.
An integral curve for the Hamiltonian vector field $
\vec{H}_f$ is a curve $\Lambda(t)=(q(t),\lambda(t))$ satisfying
\begin{align}\label{Hamiltonian Equation Vector}
    \dot\Lambda(t)=\vec H_f(\Lambda(t)).
\end{align}
A curve $\gamma:[0,T]\rightarrow \mathbb{R}^3$ with a lift $\Lambda(t)$ satisfying \eqref{Hamiltonian Equation Vector} is called a \emph{normal trajectory}. The $\lambda(t)$ component is called the \emph{covector} of the lift $\Lambda(t)$ and $\lambda(0)$ the \emph{initial covector}. Physically, $\lambda(t)$ represents the momentum of a particle whose position is described by $q(t)$, evolving according to $\eqref{Hamiltonian Equation Vector}$. 

The usual Pontryagin Maximum Principle suggests that length
minimizers should admit lifts to integral curves of the Hamiltonian vector field or lifts satisfying a certain \emph{abnormality condition}. Its standard smooth formulations, however, are not directly
applicable at $\Sigma$ under our hypotheses, since $f$ itself need
not be differentiable there.  

In the present model this difficulty
can be avoided. For the one-dimensional Sobolev facts used below, see
\cite[Chapter 8, \S8.2, in particular Theorem 8.2]{Brezis2011}.
For the basic Fr\'echet differential calculus on Banach spaces and
the differentiability of Nemytskii operators, see
\cite[Chapter 1, \S\S1.1--1.2]{AmbrosettiProdi1993}.
\begin{proposition}
\label{prop:Hamilton}
Let $f\in \mathfrak{F}$. Let $\gamma=(x,y,z):[0,T]\to\R^3$ be a nonconstant length minimizer in $(\mathbb{R}^3,d_{CC})$ which is 
parametrized with constant speed. Then there exists $w\in\R$ such that
\begin{equation}\label{eq:normal-system}
    \ddot x=-w^2A(r)x,\qquad
    \ddot y=-w^2A(r)y,\qquad
    \dot z=wf(r)^2.
\end{equation}
Consequently, setting $u=\dot x$ and $v=\dot y$, the curve
$(x,y,z,u,v,w)$ is a Hamiltonian lift of $\gamma$ for $H_f$.
\end{proposition}

\begin{proof}
Let $(a_1,a_2,a_3)$ be the minimal control of $\gamma$. Since $\gamma$ has constant speed, putting 
\[
|\dot\gamma(t)|_f :=
\begin{cases}
\displaystyle
\sqrt{\dot x(t)^2+\dot y(t)^2+
\frac{\dot z(t)^2}{f(r(t))^2}},
& r(t)>0,\\[1.2ex]
\sqrt{\dot x(t)^2+\dot y(t)^2},
& r(t)=0,
\end{cases}
\]
the curve $\gamma$ also minimizes the energy
\[
    \mathcal E(\gamma)
    =\frac12\int_0^T |\dot{\gamma}(t)|_f^2\,\dd t
\]
among competitors also defined on $[0,T]$.
Indeed, by Cauchy--Schwarz,
\[
    \mathcal E(\widetilde\gamma)
    \geq \frac{\ell(\widetilde\gamma)^2}{2T}
    \geq \frac{\ell(\gamma)^2}{2T}
    =\mathcal E(\gamma)
\]
for every competitor $\widetilde\gamma$. Write $\xi=(x,y)$ and $I(\xi):=\int_0^T f(|\xi(t)|)^2\,\dd t$. Since $\gamma$ is nonconstant, $I(\xi)>0$: otherwise $\xi\equiv0$,
and admissibility would also give $\dot z=0$ almost everywhere.

For a fixed planar curve $\xi$, the vertical endpoint condition is
\[
    z(T)-z(0)
    =
    \int_0^T f(|\xi(t)|)a_3(t)\,\dd t.
\]
Hence
\[
    |z(T)-z(0)|^2
    \leq
    I(\xi)\int_0^T a_3(t)^2\,\dd t,
\]
with equality precisely when $a_3=w f(|\xi|)$ for some constant $w$.
Replacing the vertical control by the minimizing one while keeping
$\xi$ fixed preserves the endpoints, so equality must hold for
$\gamma$. Thus
\begin{equation}\label{eq:w-definition}
    w=\frac{z(T)-z(0)}{I(\xi)},
    \qquad
    \dot z=wf(r)^2.
\end{equation}

Recall that
\[
    \Phi(\xi)=f(|\xi|)^2\in C^1(\R^2),
    \qquad
    \nabla\Phi(\xi)=2A(|\xi|)\xi.
\]
The planar component $\xi$ therefore minimizes
\begin{equation}\label{eq:reduced-energy}
    \mathcal J(\xi)
    =
    \frac12\int_0^T|\dot\xi(t)|^2\,\dd t
    +
    \frac{(z(T)-z(0))^2}{2I(\xi)}
\end{equation}
over $\mathcal A\cap\mathcal U$, where
\[
    \mathcal A
    =
    \{\eta\in W^{1,2}([0,T];\R^2):
      \eta(0)=\xi(0),\ \eta(T)=\xi(T)\},
    \qquad
    \mathcal U=\{\eta:I(\eta)>0\}.
\]

We justify the first variation. Since
$W^{1,2}([0,T];\R^2)\hookrightarrow C^0([0,T];\R^2)$ continuously,
if $h\to0$ in $W^{1,2}$ then $h\to0$ uniformly.  For $h$ sufficiently
small, the images of $\xi$ and $\xi+h$ lie in a common compact set
$K\subset\R^2$.  Using
\[
\Phi(\xi+h)-\Phi(\xi)-\langle\nabla\Phi(\xi),h\rangle
=
\int_0^1
\langle
\nabla\Phi(\xi+sh)-\nabla\Phi(\xi),h
\rangle\,\dd s,
\]
and the uniform continuity of $\nabla\Phi$ on $K$, we obtain using $\omega_K(\cdot)$ for the modulus of continuity of $\nabla\Phi$,
\[
\left|
I(\xi+h)-I(\xi)
-\int_0^T\langle\nabla\Phi(\xi),h\rangle\,\dd t
\right|
\leq
T\,\omega_K(\|h\|_\infty)\|h\|_\infty
=
o(\|h\|_{W^{1,2}}).
\]
Thus $I$ is Fr\'echet differentiable, with
\[
    DI(\xi)[h]
    =
    \int_0^T\langle\nabla\Phi(\xi),h\rangle\,\dd t.
\]
In particular, $\mathcal U$ is open and $\mathcal J$ is Fr\'echet
differentiable there.

For any $h\in W^{1,2}_0([0,T];\R^2)$, the variation
$\xi+\varepsilon h$ belongs to $\mathcal A\cap\mathcal U$ for
$|\varepsilon|$ sufficiently small. Since $\xi$ minimizes
$\mathcal J$,
\[
    D\mathcal J(\xi)[h]=0.
\]
Using \eqref{eq:w-definition} and
$\nabla\Phi(\xi)=2A(r)\xi$, this gives
\begin{equation}\label{eq:weak-horizontal-EL}
    \int_0^T
    \langle\dot\xi,\dot h\rangle\,\dd t
    =
    w^2
    \int_0^T
    A(r)\langle\xi,h\rangle\,\dd t,
    \qquad
    h\in W^{1,2}_0([0,T];\R^2).
\end{equation}
Hence
\[
    \ddot\xi=-w^2A(r)\xi
\]
in the \emph{sense of distributions}. Since $A(r)\xi$ is continuous, it
follows that $\xi\in C^2$ and the equation holds pointwise. Together
with \eqref{eq:w-definition}, this is exactly
\eqref{eq:normal-system}, and therefore
$(x,y,z,\dot x,\dot y,w)$ satisfies Hamilton's equations \eqref{Hamiltonian Equation Vector} for $H_f$.
\end{proof}
Although $\vec H_f$ is only continuous on $T^*\Sigma$, uniqueness of
Hamiltonian trajectories is not lost there. Away from $\Sigma$ this
follows from the $C^1$ regularity of $\vec H_f$. If a trajectory starts
on $\Sigma$ with vanishing planar momentum, conservation of $H_f$
gives the stationary solution. Otherwise, as shown in
Section~\ref{sec:ang mom}, the trajectory remains in the vertical
plane determined by its initial planar momentum and reduces, in signed
radial coordinates, to a one-dimensional conservative system. Its
energy identity determines the motion uniquely on each monotone branch,
through crossings of $\Sigma$, and across its turning points. Thus the
Hamiltonian equations generate a unique flow on $T^*\R^3$. 

If instead $f\in\mathfrak F_R$, then $\Phi(\xi)=f(|\xi|)^2$ belongs to
$C^2(\mathbb R^2)$ by definition.  Hence $H_f\in C^2(T^*\mathbb R^3)$
and $\vec H_f\in C^1(T^*\mathbb R^3)$.  In this case the Hamiltonian
flow depends $C^1$ on its initial condition, which will be key for later analysis.

We note that the model in question immediately rules out non-trivial \emph{abnormal extremals}, which are a source of persistent difficulty in the field of sub-Riemannian geometry. See \cite[Chapter 3]{AgrachevBarilariBoscainBook2020}. Furthermore, up to later results, we have the following. 
\begin{proposition}\label{normal are geodesic}
Every nonstationary normal trajectory is a geodesic.
\end{proposition}

\begin{proof}
If $\gamma(t)\notin\Sigma$, local optimality follows from the standard
local minimizing property of normal extremals
\cite[Theorem~4.65]{AgrachevBarilariBoscainBook2020}, whose proof applies
here away from $\Sigma$, where $H_f$ is $C^2$.

If instead $\gamma(t)\in\Sigma$, then by uniqueness of the Hamiltonian
flow, the time-shifted trajectory $s\mapsto\gamma(t+s)$ is the normal
trajectory issued from $\gamma(t)$ with the corresponding initial covector.
Theorem~\ref{singular point optimal synth} shows that such a trajectory
minimizes on a nontrivial initial interval.
\end{proof}

\subsection{Hamiltonian Equations and Cylindrical Coordinates}

With the understanding that eventually Proposition~\ref{normal are geodesic} will be proved, we use the term
\emph{normal geodesic} for a normal trajectory
$\gamma$ admitting a lift $\Lambda=(\gamma,\lambda)$ satisfying
\eqref{Hamiltonian Equation Vector}.

Hamilton's equations are
\begin{equation}\label{Hamiltonian System}
\begin{aligned}
    \dot x &= u, &
    \dot u &= -w^2A(r)x,\\
    \dot y &= v, &
    \dot v &= -w^2A(r)y,\\
    \dot z &= f(r)^2w, &
    \dot w &= 0.
\end{aligned}
\end{equation}
Thus $w=w_0$ is constant along every normal geodesic.

Throughout the paper the base point $q_0$ will be fixed and the
initial covector $\lambda_0$ varied.  All derivatives with respect to
the initial covector which are used in the Riemannian synthesis are
taken for profiles in $\mathfrak F_R$.  For the singular synthesis,
which is valid for the larger class $\mathfrak F$, the parameter
derivatives are obtained directly from \emph{radial quadratures}, which are the travel time integral and the vertical displacement integral arising from \eqref{Hamiltonian System} and the energy identity \eqref{energy identity}, rather than from parameter differentiability of the
Hamiltonian flow across $\Sigma$. 

The Euclidean coordinates $(u,v,w)$ for covectors are not well adapted to the problem. Moreover, rotational invariance suggests using cylindrical coordinates $(r,\theta,z)$ for the spatial variable $q$ instead. We introduce now the reduction of the system and the \emph{momentum coordinates} on the cotangent space $T^*_{q_0}\mathbb{R}^3$. In the following, for the function $t\mapsto r(t)=\sqrt{x^2(t)+y^2(t)}$, we will use the language \emph{radial trajectory} or \emph{radial motion}.
 
Along every normal geodesic the Hamiltonian $H_f$ has a constant value, which we denote by $E$. The \emph{speed} of such a geodesic is $c=\sqrt{2E}$. Furthermore, the \emph{angular momentum}
\[
    K:=xv-yu
\]
is constant. Indeed, the conservation of $K$ follows immediately
from \eqref{Hamiltonian System}. Furthermore, a short computation shows that for all $r>0$, $K=r^2\dot\theta$. Thus,
\begin{equation}\label{eq:radial-energy-general}
    \dot r^2+\frac{K^2}{r^2}+w^2f(r)^2=2E.
\end{equation}
Equivalently, the radial motion has \emph{effective potential}
\[
    V_{K,w}(r)
    :=
    \frac{K^2}{2r^2}+\frac{w^2}{2}f(r)^2.
\]
As a second order ODE, this becomes $-w_0^2 f(r)\dot{f}(r)=\ddot{r}-\dot{\theta}^2r$. Then, using the reduction $K=r^2\dot{\theta}$, we obtain the second order ODE when $r\neq 0$,
\begin{align}\label{Radial ODE full}
   \ddot{r}=-w_0^2f(r)\dot{f}(r)+\frac{K^2}{r^3}. 
\end{align}
We note that $p_r=\dot{r}$ physically describes the \emph{radial momentum} of a particle evolving according to the Hamiltonian system. The radially scaled initial value of the radial momentum $L:=\dot{r}(0)r(0)=x_0u_0+y_0v_0$ when $r(0)\neq 0$ is such that $(K,L,w)$ forms an orthogonal coordinate system on the cotangent space, and we refer to this coordinate system as \emph{momentum coordinates}. 

Equation \eqref{Radial ODE full} exhibits three qualitatively distinct regimes. When $K=0$, since $\dot{\theta}$ then vanishes, the motion is strictly radial and lies in a fixed vertical plane through the axis. This reduced system exhibits 2D-Grushin style dynamics. In particular, geodesics with $K=0$ eventually cross the singular set
$\Sigma=\{r=0\}$. When $K\neq 0$, the centrifugal term $K^2/r^3$ prevents collision with $\Sigma$. We note that when $w=0$, the geodesic reduces to a planar straight line with constant vertical coordinate $z\equiv z_0$. It is clear that the planar straight lines are length minimizers for $d_{CC}$. 

See Figure \ref{fig:3 regimes}.
\subsection{Reduction to Planar and Radial Dynamics}\label{sec:ang mom}
For $K=0$, write $(x(t),y(t))=\rho(t)e$ for a fixed
$e\in\mathbb S^1$, allowing $\rho$ to change sign, and put
$F(\rho):=f(|\rho|)^2$. By \ref{A2}, $F\in C^1(\R)$, and
\eqref{Radial ODE full} becomes, in \emph{signed radial coordinates}
$(\rho,z)$,
\begin{equation}\label{eq:signed-radial-system}
    \ddot\rho=-\frac{w^2}{2}F'(\rho),
    \qquad
    \dot z=wF(\rho).
\end{equation}
There is an ambiguity in the orientation of the signed radial
coordinate, which we fix as follows. If $q_0\notin\Sigma$, we choose
$e=(x_0,y_0)/r_0$, where $r_0=\sqrt{x_0^2+y_0^2}$, so that
\[
    \rho(0)=r_0>0,
    \qquad
    \dot\rho(0)=\frac{L}{r_0}.
\]
If $q_0\in\Sigma$ and the geodesic is nonstationary, we orient $e$
in the direction of the initial planar velocity, determined uniquely by $(u(0),v(0))$, so that
\[
    \rho(0)=0,
    \qquad
    \dot\rho(0)=\sqrt{2E}>0.
\]
Thus the signed radial coordinate is always oriented so that the
initial point, or at a singular initial point the initial motion itself,
points in the positive $\rho$ direction. In the context of the signed radial coordinates, we will refer to $|\rho|$ as the \emph{physical radius}.

We note that \eqref{eq:signed-radial-system} admits the first integral
\begin{equation}\label{energy identity}
    \dot\rho^2+w^2F(\rho)=2E,
\end{equation}
which is just \eqref{eq:radial-energy-general} with $K=0$ and written in signed radial coordinates.

A \emph{radial turn} is a time at which $\dot\rho=0$ when $K=0$, or
$\dot r=0$ when $K\neq0$. The corresponding value of the physical
radius is called a \emph{turning radius}. Thus, when $K=0$, a positive
turning radius $r_*$ corresponds to the two signed turning values
$\rho=\pm r_*$, while for $K\neq0$ the turning radii are values of
$r$ itself.

We now classify the radial turns and turning radii when $K=0$. If $E=0$, the trajectory is stationary. If $E>0$ and $w=0$, then
\eqref{energy identity} gives $\dot\rho^2=2E$ and the geodesic is a planar straight line, so there are no radial
turns. Suppose therefore that $E>0$ and $w\neq0$. Since $f$ is strictly increasing and unbounded by hypothesis, there is a unique $r_*>0$ such that
\[
    f(r_*)=\frac{\sqrt{2E}}{|w|}.
\]
By \eqref{energy identity}, the radial turns occur precisely at
$\rho=\pm r_*$, and hence $r_*$ is the unique turning radius. The
signed radial motion oscillates between $-r_*$ and $r_*$. The \emph{radial period}, i.e. the least positive
period of $r(t)=|\rho(t)|$, is obtained as a \emph{travel time integral} through \eqref{energy identity} as
\begin{equation}\label{Period Formula}
    \mathbf P(E,0,w)
    :=
    2\int_0^{r_*}
    \frac{\dd r}{\sqrt{2E-w^2f(r)^2}}=\frac{2r_*}{\sqrt{2E}}\int_0^1\frac{ds}{\sqrt{1-\left(\frac{f(r_*s)}{f(r_*)}\right)^2}}.
\end{equation}
The second equality follows by a change of variables. This representation will be useful later when studying the dependence of $\mathbf P$ on the turning radius $r_*$ and on $w$. Here we write $0$ as an argument because we will eventually replace it with $K$, which will be allowed to vary.

If $r_0=0$, with $P=\mathbf P(E,0,w)$, note that $\rho(P/2)=r_*$, $\rho(P)=0$, $\dot\rho(P)=-\sqrt{2E}$, and since the signed radial equation is autonomous and invariant under
$\rho\mapsto-\rho$, uniqueness gives
\[
    \rho(P-t)=\rho(t),\qquad
    \rho(t+P)=-\rho(t),\qquad
    \rho(t+2P)=\rho(t).
\]
In particular, the radial trajectory has an \emph{outgoing branch} $(0,P/2)$, and an \emph{incoming} or \emph{returning} branch on $(P/2,P)$. 

When $f(r)=r^\alpha$, formula \eqref{Period Formula} becomes explicit. Writing $\pi_\alpha:=2\int_0^1\frac{dt}{\sqrt{1-t^{2\alpha}}}
=\frac1\alpha B\left(\frac1{2\alpha},\frac12\right)$, with $B$ the complete Beta function, we obtain $\mathbf P(E,0,w_0)=\pi_\alpha(2E)^{\frac1{2\alpha}-\frac12}|w_0|^{-1/\alpha}$ and in the unit-speed case $2E=1$ this reduces to $\mathbf P=\pi_\alpha|w_0|^{-1/\alpha}$. This is the exact \emph{cut time} for geodesics from singular points in the $\alpha$-Grushin plane first discovered in \cite{Borza2022}. These constants and the associated generalized trigonometric functions have been studied in \cite{albert2025geodesicsgrushinspaces,Borza2022,MagnaboscoRossi2023}; for further properties and extensions of generalized trigonometric functions, see \cite{KobayashiTakeuchi2019,Lokutsievskii2019}.

We now consider $K\neq0$. Let us further specialize to $Kw\neq0$, because $w=0$ corresponds to the straight line geodesics. In this case
\[
    V_{K,w}(r)\longrightarrow+\infty
    \qquad\text{as }r\downarrow0
    \quad\text{or}\quad r\to\infty.
\]
Nevertheless, this \emph{coercivity} does not imply that every nonconstant
radial trajectory is periodic: the effective potential may have
several \emph{wells} or a critical energy level containing a \emph{separatrix}. We call a normal geodesic \emph{radially periodic} if its physical
radius $r(t)$ is a nonconstant periodic function. When $K=0$, we
understand $r(t)=|\rho(t)|$ in signed radial coordinates. The least
positive period of $r$ is called the \emph{radial period}.

For $K\neq0$, if $r_-<r_+$ are the consecutive turning radii bounding
the radial trajectory, then
\begin{equation}\label{eq:general-radial-period}
    \mathbf P(E,K,w)
    :=
    2\int_{r_-}^{r_+}
    \frac{\dd r}
    {\sqrt{2E-K^2/r^2-w^2f(r)^2}}.
\end{equation}
For $K=0$, the period is given by \eqref{Period Formula}.
The quantity $\Psi(r):=r^3f(r)f'(r)$ from the definition of $\mathfrak{F}_R^{\mathbf{P}}$
has a dynamical interpretation. Indeed, when $Kw\neq0$,
\[
    V_{K,w}'(r)
    =
    \frac{w^2}{r^3}
    \left(
        \Psi(r)-\frac{K^2}{w^2}
    \right).
\]
Thus intersections of $\Psi$ with the level $K^2/w^2$
are precisely the critical radii of the effective potential.  If
$\Psi$ is nondecreasing, the effective potential has a single well;
if $\Psi$ decreases somewhere, one can instead produce a local
maximum and hence a separatrix.  This observation gives the following
criterion for radial periodicity and will reappear in Section~7.

In the following we say that a geodesic is \emph{radially constant} or \emph{circular} if $r\equiv r_0$. 
\begin{proposition}
\label{prop:radial-periodicity}
Set $\Psi(r):=r^3f(r)f'(r)$ for $r>0$. Then the following are equivalent:
\begin{enumerate}
    \item $\Psi$ is nondecreasing on $(0,\infty)$.
    \item Every normal geodesic is either radially periodic, radially constant,
    or a planar straight line.
\end{enumerate}
\end{proposition}

\begin{proof}
Suppose first that $\Psi$ is nondecreasing. If $w=0$, the geodesic is a
planar straight line. Suppose next that $K=0$ and $w\neq0$. We already
saw that the physical radius $r=|\rho|$ is periodic. It remains to
consider $Kw\neq0$.

By \ref{A2}, $\Psi(r)=r^4A(r)\longrightarrow0$
as $r\downarrow0$. Moreover, $\Psi(r)\to+\infty$ as $r\to+\infty$. Indeed, if $\Psi$ were
bounded, then $f(r)f'(r)\lesssim r^{-3}$ for large $r$, and integration
of $(f^2)'=2ff'$ would imply that $f$ is bounded, contradicting
\ref{A3}. Fix $K,w\neq0$. Since $\Psi$ is continuous and nondecreasing, with
\[
    \lim_{r\downarrow0}\Psi(r)=0,
    \qquad
    \lim_{r\to+\infty}\Psi(r)=+\infty,
\]
the level set
\[
    \Psi^{-1}\!\left(\frac{K^2}{w^2}\right)
    =[r_1,r_2]
\]
is a nonempty closed interval, possibly reduced to a single point. The effective potential satisfies
\[
    V_{K,w}'(r)
    =
    \frac{w^2}{r^3}
    \left(\Psi(r)-\frac{K^2}{w^2}\right).
\]
Hence $V_{K,w}'(r)<0$ for $0<r<r_1$, and $V_{K,w}'(r)=0$
for $r_1\leq r\leq r_2$, 
and $V_{K,w}'(r)>0$ for $r>r_2$. 
Thus $V_{K,w}$ has a global minimum, attained precisely on
$[r_1,r_2]$.

If $E=\min V_{K,w}$
then the radial energy identity
\[
    \dot r^2=2\bigl(E-V_{K,w}(r)\bigr)
\]
forces $\dot r\equiv0$, and hence the geodesic is radially constant.

If instead $E>\min V_{K,w}$,
then, since $V_{K,w}(r)\longrightarrow+\infty$ as 
    $r\downarrow0$ and as $r\to+\infty$, the equation $V_{K,w}(r)=E$
has exactly two roots
\[
    r_-<r_1\leq r_2<r_+.
\]
Both roots are simple. The radial trajectory oscillates between
$r_-$ and $r_+$ and is therefore periodic, with radial period given by
\eqref{eq:general-radial-period}. This proves $(1)\Rightarrow(2)$.

Conversely, suppose that every normal geodesic is either radially
periodic, radially constant, or a planar straight line. We show that
$\Psi$ must be nondecreasing.

Since $f\in C^2(0,\infty)$, the function $\Psi$ belongs to
$C^1(0,\infty)$. If $\Psi$ were not nondecreasing, then by the mean value
theorem there would exist $r_0>0$ such that $\Psi'(r_0)<0$.

Since $\Psi\geq0$, necessarily $\Psi(r_0)>0$. Choose $w\neq0$ and $K\neq0$
so that
\[
    \frac{K^2}{w^2}=\Psi(r_0).
\]
Then
\[
    V_{K,w}'(r)
    =
    \frac{w^2\bigl(\Psi(r)-\Psi(r_0)\bigr)}{r^3},
\]
and therefore $V_{K,w}'(r_0)=0$.
Moreover,
\[
    V_{K,w}''(r_0)
    =
    \frac{w^2\Psi'(r_0)}{r_0^3}<0.
\]
Thus $r_0$ is a strict nondegenerate local maximum of $V_{K,w}$.

Set $E:=V_{K,w}(r_0)$. For $r_*>r_0$ sufficiently close to $r_0$, the critical energy level
contains the radial branch
\[
    \dot r
    =
    -\sqrt{2\bigl(E-V_{K,w}(r)\bigr)}.
\]
Since the maximum is nondegenerate, $E-V_{K,w}(r)
    \asymp
    (r-r_0)^2$ as $r\downarrow r_0$. Consequently, computing the travel time integral between $r_0$ and $r_*$, we have
\[
    \int_{r_0}^{r_*}
    \frac{dr}
    {\sqrt{2\bigl(E-V_{K,w}(r)\bigr)}}
    =
    +\infty.
\]
Hence the corresponding radial trajectory decreases monotonically
toward $r_0$ and reaches it only as $t\to+\infty$. It is therefore
neither radially periodic nor radially constant. Since $w\neq0$, it is
not a planar straight line either. This contradicts $(2)$.

Therefore $\Psi$ is nondecreasing, proving $(2)\Rightarrow(1)$.
\end{proof}
\begin{remark}
One may check that for $f$ given by the following, $f\in \mathfrak{F}$ and yet the corresponding $\Psi(r)$ is not non-decreasing, hence by Proposition~\ref{prop:radial-periodicity} there is a normal geodesic which is not radially periodic:
\[
    f(r)=\frac{r^2}{(1+r^{40})^{1/40}}\qquad \eta(r)=\frac{rf'(r)}{f(r)}
    =\frac{2+r^{40}}{1+r^{40}}.
\]
\end{remark}
\subsection{Continuous Radial Period}
We now package the preceding radial dynamics into a single
extended-valued function on the cotangent bundle. Fix
$q_0=(x_0,y_0,z_0)$ and, for
$\lambda=(u,v,w)\in T_{q_0}^*\R^3$, put $E=H_f(q_0,\lambda)$ and 
$K=x_0v-y_0u$.
When $r_0>0$, recall also that
\[
    L=r_0\dot{r}(0),
    \qquad
    2E=\frac{L^2+K^2}{r_0^2}+w^2f(r_0)^2.
\]
Thus the quantity defined below depends on the momentum coordinates
$(L,K,w)$ through $(E,K,w)$.

Assume henceforth that $f\in \mathfrak{F}_R^{\mathbf{P}}$, which is stronger than what Proposition \ref{prop:radial-periodicity} requires. In particular, there is a distinct value $r_c$ at which $V_{K,w}(r)$ reaches its minimum. We define
\[
    \mathbf P_{q_0}(\lambda)
    :=
    \mathbf P(E,K,w)\in(0,+\infty]
\]
as follows, retaining the convention that for $K=0$ this is the period of the physical radius $r(t)=|\rho(t)|$. If $E=0$ or $w=0$, set
\[
    \mathbf P(E,K,w):=+\infty.
\]
If $E>0$, $w\neq0$, and $K=0$, let $r_*>0$ be determined by
$f(r_*)=\frac{\sqrt{2E}}{|w|}$ and define $\mathbf P(E,0,w)$ by \eqref{Period Formula}.
Finally, suppose $Kw\neq0$, and let $r_c$ be the unique solution of $r_c^3f(r_c)f'(r_c)=\frac{K^2}{w^2}$. If $E>V_{K,w}(r_c)$,
let $r_-<r_c<r_+$ be the two solutions of
$V_{K,w}(r)=E$ and define $\mathbf P(E,K,w)$ by \eqref{eq:general-radial-period}.
If instead $E=V_{K,w}(r_c)$, so that the corresponding geodesic is
radially constant, we define $\mathbf P$ by continuous extension from
the radially periodic geodesics:
\begin{equation}\label{eq:P-circular-extension}
    \mathbf P(E,K,w)
    :=
    \lim_{\widetilde E\downarrow E}
    \mathbf P(\widetilde E,K,w).
\end{equation}
When $V_{K,w}''(r_c)>0$, this gives
\[
    \mathbf P(E,K,w)
    =
    \frac{2\pi}{\sqrt{V_{K,w}''(r_c)}}
    =
    \frac{2\pi r_c^{3/2}}
    {|w|\sqrt{\Psi'(r_c)}}
\]
while if $V_{K,w}''(r_c)=0$ the extended value is $+\infty$. In fact, the following is true by a standard change of variables, wherein the singularities in the integral defining $\mathbf{P}$ are seen to be of square root type. 
\begin{theorem}\label{lem:radial-period-continuity}
Let $f\in \mathfrak{F}_R^{\mathbf{P}}$. For
each $E>0$, the radial-period function
\[
    (q,\lambda)\longmapsto
    \mathbf P_q(\lambda)
    =\mathbf P(E,K(q,\lambda),w)
\]
is continuous on the energy level $\{(q,\lambda)\in T^*\R^3:H_f(q,\lambda)=E\}$
as a map into $(0,+\infty]$ equipped with the extended topology. In
particular, the $K=0$ formula is the continuous limit as $K\to0$, the
value at a radially constant orbit is the limit of the periods of
nearby periodic radial orbits, and $\mathbf P(E,K,w)\longrightarrow+\infty$ as $w\to 0$.
\end{theorem}
With the understanding that $\mathbf{P}$ will be turned into a cut mechanism, Theorem~\ref{lem:radial-period-continuity} should be seen as an analogue of \cite[Proposition 8.76]{AgrachevBarilariBoscainBook2020}, which states that in ideal, pure sub-Riemannian structures, the true \emph{cut time} is a continuous function on the subspace of the cotangent bundle given by fixed level sets (energy levels) of the Hamiltonian function $H_f$.
\begin{figure}
    \centering
    \includegraphics[width=0.5\linewidth]{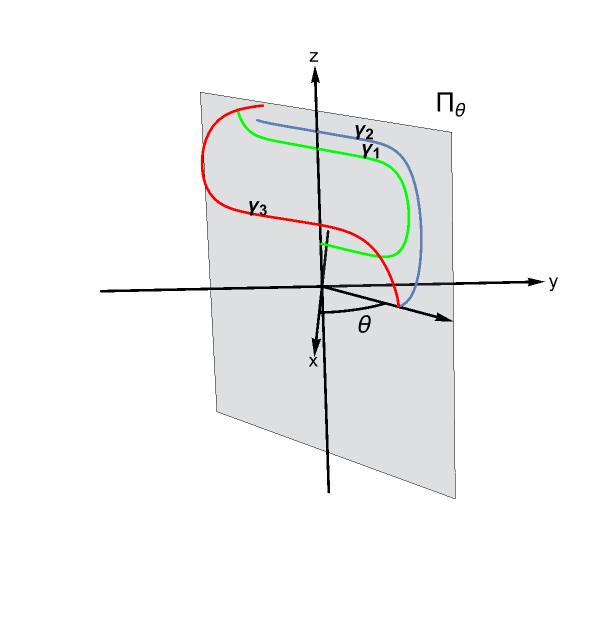}
    \caption{Three regimes of geodesics: $\gamma_1$ starts from a singular point. $\gamma_2$ starts from a Riemannian point with $K=0$ and remains in a fixed vertical plane through the axis. $\gamma_3$ starts from a Riemannian point with $K\neq0$ and leaves that plane before returning on the opposite side of $\Sigma$.}
    \label{fig:3 regimes}
\end{figure}
\section{Optimal Synthesis}
We begin by recalling the notions of cut and conjugate time.
\begin{definition}
Let $\gamma:[0,\infty)\to\R^3$ be a geodesic with
$q_0=\gamma(0)$. Its \emph{cut time} is
\[
    t_{\operatorname{cut}}(\gamma)
    :=
    \sup\{t>0:\gamma|_{[0,t]}
    \text{ is a length minimizer}\}.
\]
If $t_{\operatorname{cut}}(\gamma)<+\infty$, the point
$\gamma(t_{\operatorname{cut}}(\gamma))$ is a \emph{cut point} of
$q_0$. The collection of all cut points of $q_0$ is its
\emph{cut locus}, denoted $\operatorname{Cut}(q_0)$. For $\lambda\in T^*_{q_0}\R^3$, define $\operatorname{Exp}_{q_0}(\lambda)
    :=
    \gamma(1;\lambda)$,
where $\gamma(\cdot;\lambda)$ is the normal geodesic with initial
covector $\lambda$. If $\lambda_0$ is the initial covector of
$\gamma$, its \emph{first conjugate time} is
\[
    t_{\operatorname{con}}(\gamma)
    :=
    \inf\{t>0:
    t\lambda_0\text{ is a critical point of }
    \operatorname{Exp}_{q_0}\}.
\]
The corresponding point is said to be \emph{conjugate} to $q_0$.
\end{definition}
A standard way to obtain an upper bound for the cut time is through the construction of
\emph{Maxwell points}: two distinct geodesics of the same constant
speed, issued from the same point, meet after the same positive time.
The first such meeting time is a candidate cut time. Establishing that
it is the true cut time requires ruling out loss of optimality at an
earlier time by an earlier intersection with a third geodesic or by premature conjugacy. 

For initial Riemannian points, often the best strategy is to employ an Extended Hadamard technique \cite[Chapter 13]{AgrachevBarilariBoscainBook2020}. The standard proof is stated for complete, ideal, pure sub-Riemannian structures, but can be extended to our setting without issue, as it simply goes through the covering map theory for $C^1$ functions between manifolds as applied to the exponential $\operatorname{Exp}_{q_0}$. We have shown completeness already in Theorem~\ref{metric}. One needs to ensure that the conjectured cut time is continuous in the extended sense, so that the conjectured injectivity domain is a manifold to begin with. We have done so with Theorem~\ref{lem:radial-period-continuity}. Furthermore, it is essential here that $f\in \mathfrak{F}_R$, as the $K=0$ geodesics started from a Riemannian point will eventually cross the singular set $\Sigma$. We do this so that the exponential map is actually $C^1$ on the conjectured injectivity domain. 
\begin{theorem}[Extended Hadamard technique]
\label{extended Hadamard}
Assume $f\in \mathfrak{F}_R$. Let $q_0\notin\Sigma$ and set
\[
    E_1
    :=
    \{\lambda\in T^*_{q_0}\R^3:
      H_f(q_0,\lambda)=1/2\}.
\]
Let $t_*:E_1\to(0,+\infty]$ be continuous in the extended sense and suppose that
$t_{\operatorname{cut}}(\lambda)\le t_*(\lambda)$ for every $\lambda\in E_1$. Define the conjectured cotangent injectivity domain by
$\mathcal{N}
    :=
    \{t\lambda:\lambda\in E_1,\ 0\le t<t_*(\lambda)\}$
and $\operatorname{Cut}^*(q_0)
    :=
    \{
      \operatorname{Exp}_{q_0}(t_*(\lambda)\lambda):
      \lambda\in E_1,\ t_*(\lambda)<+\infty
    \}$. Suppose that:
\begin{enumerate}[label=\textnormal{(EHT\arabic*)}]
    \item\label{proper EHT}
    $\operatorname{Exp}_{q_0}|_{\mathcal{N}}:
    \mathcal{N}\to\R^3\setminus\operatorname{Cut}^*(q_0)$ is proper;

    \item\label{loc diffeo EHT}
    $t_*(\lambda)\le t_{\operatorname{con}}(\lambda)$ for every
    $\lambda\in E_1$, so that
    $\operatorname{Exp}_{q_0}|_\mathcal{N}$ is a local diffeomorphism;

    \item\label{inclusion EHT}
$\operatorname{Exp}_{q_0}(\mathcal{N})
        =
        \R^3\setminus\operatorname{Cut}^*(q_0)$
    \item\label{simply conn EHT}
    $\R^3\setminus\operatorname{Cut}^*(q_0)$ is simply connected.
\end{enumerate}
Then $t_*=t_{\operatorname{cut}}$
on $E_1$, and consequently
$\operatorname{Cut}^*(q_0)=\operatorname{Cut}(q_0)$.
\end{theorem}
Given some positive, finite valued function $t_*$ on a subset
$D\subseteq E_1$, we call
\[
    D\ni\lambda\mapsto\operatorname{Exp}_{q_0}\bigl(t_*(\lambda)\lambda\bigr)\in\mathbb{R}^3,
\]
the associated \emph{endpoint map}. We will often express this map
in parameters adapted to the initial covector, rather than in
$\lambda$ itself. In particular, when $t_*=\mathbf P$ is the
full radial period constructed in Theorem~\ref{lem:radial-period-continuity}, we call the associated endpoint map the
\emph{full-period endpoint map}, denoted $\mathcal R$ when the base point is fixed. When $q_0\in \Sigma$, we will sometimes call this the \emph{first return endpoint map}. 
\subsection{Optimal Synthesis at Singular Points}

We first determine the optimal synthesis for geodesics issued from points on $\Sigma$. In this case $K=0$ for every initial covector, and the radial
period $\mathbf P(E,0,w)$ introduced above will turn out to be exactly
the cut time. We follow the shooting strategy used for the classical Grushin
plane and the $\alpha$-Grushin plane in \cite{AgrachevBarilariBoscainBook2020,Borza2022}. 

\begin{theorem}\label{singular point optimal synth}
Let $f\in \mathfrak{F}$. Let $q_0=(0,0,z_0)\in\Sigma$, and let $\gamma$ be a nonstationary
normal geodesic issued from $q_0$ with initial covector
$\lambda_0=(u_0,v_0,w)$ and energy $E>0$. If $w=0$, then $\mathbf P(E,0,0)=+\infty$ and $\gamma$ is a globally
minimizing planar straight line. If $w\neq0$, let $r_*>0$ be the
unique turning radius determined by $f(r_*)=\frac{\sqrt{2E}}{|w|}$. Then
\[
    t_{\operatorname{cut}}(\gamma)
    =
    \mathbf P_{q_0}(\lambda_0)
    =
    \mathbf P(E,0,w)= 2\int_0^{r_*}
    \frac{\dd r}{\sqrt{2E-w^2f(r)^2}}
\]
At this time $\gamma$ returns to $\Sigma$ and meets infinitely many
minimizing geodesics of the same length. Moreover, $\operatorname{Cut}(q_0)
    =
    \Sigma\setminus\{q_0\}$.
\end{theorem}

\textbf{Outline of the Proof:} For $w=0$ the theorem is immediate, so suppose $w\neq0$.  The proof is
a shooting argument in the vertical momentum $w$, with the energy held
fixed.  First, the radial period is a Maxwell time, which gives the
upper bound $t_{\operatorname{cut}}\leq\mathbf P$.  The elasticity
condition~\ref{A4} then orders the radial motion as $w$ varies:
Lemma~\ref{singular period monotonicity} controls the period, while
Lemma~\ref{technical} transfers this monotonicity to the endpoint
coordinates.  Fixed off-axis radii are hit by a geodesic issued from $\Sigma$ twice. The corresponding vertical height on the
outgoing branch increases with $w$, whereas the height on the returning
branch decreases. We glue these height branches together and obtain their limiting values to show that any off axis point hit by a unit speed geodesic issued from $\Sigma$ must have been hit by precisely one unit speed geodesic. Completeness and Proposition~\ref{prop:Hamilton}
then finish the argument. 

\begin{lemma}\label{not minimizing past}
Let $\gamma$ be a normal geodesic issued from $q_0\in\Sigma$ with
$E>0$ and $w\neq0$. Then
\[
    t_{\operatorname{cut}}(\gamma)
    \leq \mathbf P(E,0,w).
\]
\end{lemma}

\begin{proof}
Write $P=\mathbf P(E,0,w)$. For fixed $(E,w)$, the signed radial
trajectory $\rho$, its radial period $P$, and the vertical component
\[
    z(t)=z_0+w\int_0^t F(\rho(s))\,\dd s
\]
are independent of the initial planar direction $e\in\mathbb S^1$.
Since $\rho(t)>0$ for $0<t<P$, geodesics corresponding to distinct
directions do not meet before $P$, whereas $\rho(P)=0$ shows that they
all meet at the same point of $\Sigma$ at time $P$.

Suppose, toward a contradiction, that $\gamma|_{[0,P+\varepsilon]}$
is minimizing for some $\varepsilon>0$. Choose another initial
direction $\widehat e\neq e$ and let $\widehat\gamma$ be the
corresponding geodesic. Since
$\widehat\gamma(P)=\gamma(P)$ and the two curves have the same constant
speed, replacing $\gamma|_{[0,P]}$ by
$\widehat\gamma|_{[0,P]}$ produces another minimizer from
$q_0$ to $\gamma(P+\varepsilon)$. At the joining time $P$, however, its
planar velocity jumps from
$-\sqrt{2E}\,\widehat e$ to $-\sqrt{2E}\,e$. This contradicts
Proposition~\ref{prop:Hamilton}, according to which every nonconstant
constant-speed minimizer has a $C^2$ planar component. Hence $\gamma$
cannot minimize past $P$.
\end{proof}

We next record the consequence of the elasticity condition \ref{A4} on the radial period. 

\begin{lemma}\label{singular period monotonicity}
Fix $E>0$. Then $\mathbf P(E,0,w)$ is strictly decreasing as a
function of $|w|>0$. Equivalently, when parametrized by the turning
radius $r_*$, the period is strictly increasing in $r_*$.
\end{lemma}

\begin{proof}
Recall that $\eta(r)>0$ for every $r>0$.  For $a>0$, set
$R_a(s):=f(as)/f(a)$, $0\leq s\leq1$.  If $a$ is the turning
radius, then
\[
    \mathbf P\!\left(E,0,\frac{\sqrt{2E}}{f(a)}\right)
    =
    \frac{2a}{\sqrt{2E}}
    \int_0^1\frac{\dd s}{\sqrt{1-R_a(s)^2}},
\]
and, for $0<s<1$, $a\,\partial_aR_a(s)
    =
    R_a(s)\bigl(\eta(as)-\eta(a)\bigr)$. We justify differentiation of the period with respect to $a$.  Fix
$a>0$ and let $I\Subset(0,\infty)$ be a compact neighborhood of $a$.
For $b\in I$,
\[
    \partial_b
    \left(
        \frac{b}{\sqrt{1-R_b(s)^2}}
    \right)
    =
    \frac{1}{\sqrt{1-R_b(s)^2}}
    +
    \frac{
        R_b(s)^2\bigl(\eta(bs)-\eta(b)\bigr)
    }{
        (1-R_b(s)^2)^{3/2}
    }.
\]
Since $\partial_sR_b(1)=\eta(b)>0$, after shrinking $I$ if necessary
there exist $\delta,c,C>0$ such that
\[
    c(1-s)\leq1-R_b(s)^2\leq C(1-s)
\]
for $b\in I$ and $1-\delta<s<1$.  Since
$\eta\in C^1((0,\infty))$ and is nonincreasing,
\[
    0\leq\eta(bs)-\eta(b)\leq C(1-s)
\]
on the same set.  Hence the derivative of the integrand is bounded
there by $C(1-s)^{-1/2}$.

There is no singularity at $s=0$.  Indeed, by \ref{A2}, $A$ is
bounded near $0$, and
\[
    R_b(s)^2\eta(bs)
    =
    \frac{b^2s^2A(bs)}{f(b)^2}.
\]
Thus differentiation under the integral sign is valid and gives
\[
    \frac{\dd}{\dd a}
    \mathbf P\!\left(E,0,\frac{\sqrt{2E}}{f(a)}\right)
    =
    \frac{2}{\sqrt{2E}}
    \int_0^1
    \left[
        \frac{1}{\sqrt{1-R_a(s)^2}}
        +
        \frac{
            R_a(s)^2\bigl(\eta(as)-\eta(a)\bigr)
        }{
            (1-R_a(s)^2)^{3/2}
        }
    \right]\dd s.
\]
By \ref{A4}, both terms in the integrand are nonnegative and the first
is strictly positive.  Hence the period is strictly increasing with
the turning radius $a$.  Since
$f(r_*)=\sqrt{2E}/|w|$ and $f$ is strictly increasing, $r_*$ is
strictly decreasing in $|w|$, proving the claim.
\end{proof}

The next lemma is the link between radial-period monotonicity and
endpoint injectivity.  On a radial branch, differentiating
the travel-time quadrature at fixed energy determines the sign of
$\partial_w\rho$.  Differentiating the radial energy identity then
expresses the vertical variation in terms of the same quantity. 

The derivatives with respect to $w$ used below are obtained directly
from the radial quadrature and do not require differentiable
dependence of the full Hamiltonian flow across $\Sigma$.  On
$0<t<\mathbf P(E,0,w)$ the signed radius is positive, so the
differentiations themselves take place in the region $\{r>0\}$ where the Hamiltonian flow is at least $C^1$ with respect to $w$, and the variations themselves are at least $C^1$ in time.  

\begin{lemma}\label{technical}
Let $q_0\in\Sigma$, $E>0$, and $w\neq0$. Then:
\begin{enumerate}[label=\textnormal{(\alph*)}]
    \item If $w>0$, then $\partial_w\rho(t;w)<0$ for
    $0<t<\mathbf P(E,0,w)$, while if $w<0$, then
    $\partial_w\rho(t;w)>0$ on the same interval.
    \item For $0<t<\mathbf P(E,0,w)$,
    \[
        \partial_w z(t;w)
        =
        -\frac1w\dot\rho(t;w)\partial_w\rho(t;w).
    \]
\end{enumerate}
\end{lemma}

\begin{proof}
By symmetry it suffices to prove \emph{(a)} for $w>0$.  Write
$P=\mathbf P(E,0,w)$ and
$r_*(w)=f^{-1}(\sqrt{2E}/w)$, so that $r_*'(w)<0$.
On the outgoing branch $0<t<P/2$, the travel time integral may be written as
\[
    t
    =
    \int_0^{\rho(t;w)}
    \frac{\dd s}{\sqrt{2E-w^2f(s)^2}}.
\]
The derivative of the right-hand side with respect to its upper limit
is positive, and its derivative with respect to $w$ is positive.
Implicit differentiation therefore gives
$\partial_w\rho(t;w)<0$.  At the turning time,
$\rho(P/2;w)=r_*(w)$, so
$\partial_w\rho(P/2;w)=r_*'(w)<0$.

For $P/2<t<P$, put $s=P-t$.  Since
$\rho(t;w)=\rho(s;w)$,
\[
    \partial_w\rho(t;w)
    =
    \partial_w\rho(s;w)
    +
    \dot\rho(s;w)\,\partial_w\mathbf P(E,0,w).
\]
The first term is negative, $\dot\rho(s;w)>0$, and the last factor is
negative by Lemma~\ref{singular period monotonicity}.  Hence
$\partial_w\rho(t;w)<0$ on the returning branch as well.  The case
$w<0$ follows because the radial equation depends only on $w^2$.

For \emph{(b)}, differentiate
$\dot\rho^2+w^2f(\rho)^2=2E$ with respect to $w$, and use
$\ddot\rho=-w^2f(\rho)f'(\rho)$ and $\dot z=wf(\rho)^2$.  A direct
calculation gives
\[
    \frac{\dd}{\dd t}
    \left(
        \partial_w z
        +
        \frac{\dot\rho\,\partial_w\rho}{w}
    \right)
    =0.
\]
While we can directly invoke that $\partial_w\rho(0;w)=\partial_wz(0;w)=0$ by differentiating the initial conditions in the second order ODE driving $\rho$, we cannot immediately conclude that the parameter derivatives computed above actually converge to these values, because we do not know that for $f\in \mathfrak{F}$, the Hamiltonian flow is $C^1$ in its initial covector across $\Sigma$. Rather, we may prove this convergence. Instead,
the outgoing radial quadrature gives, for sufficiently small $t>0$,
\[
    \partial_w\rho(t;w)
    =-\dot\rho(t;w)
      \int_0^{\rho(t;w)}
      \frac{w f(s)^2}
           {(2E-w^2f(s)^2)^{3/2}}\,\dd s
    \longrightarrow 0
    \qquad (t\downarrow0).
\]
Likewise, differentiating the vertical quadrature
\[
    z(t;w)-z(0)
    =\int_0^{\rho(t;w)}
      \frac{w f(s)^2}{\sqrt{2E-w^2f(s)^2}}\,\dd s
\]
shows that $\partial_w z(t;w)\to0$ as $t\downarrow0$.
Indeed, the integration interval shrinks to zero, $f(0)=0$,
and the denominators remain bounded away from zero.
Thus the constant quantity in parentheses has limit zero
as $t\downarrow0$, proving \emph{(b)}.
\end{proof}
\begin{proof}[Proof of Theorem~\ref{singular point optimal synth}]
By translation in $z$ we may assume $q_0=0$, and by quadratic homogeneity of the
Hamiltonian it is enough to work at unit speed, so $2E=1$. Write
\[
    P(w):=\mathbf P(1/2,0,w),
    \qquad
    r_*(w):=f^{-1}(1/|w|).
\]
We first prove uniqueness before the first return. Let
$\bar q=(\bar x,\bar y,\bar z)\notin\Sigma$, put
\[\bar r=(\bar x^2+\bar y^2)^{1/2}>0,\] and orient the signed radial
coordinate toward $(\bar x,\bar y)$. By reflection in $z$ it suffices
to consider $\bar z>0$, in which case necessarily $w>0$. Set
$w_*=1/f(\bar r)$. For $0<w<w_*$ there are exactly two times
$t_1(w)<P(w)/2<t_2(w)<P(w)$ such that
$\rho(t_j(w);w)=\bar r$, while at $w=w_*$ the two times coincide at
the radial turn. Implicit differentiation gives
\[
    t_j'(w)
    =
    -\frac{\partial_w\rho(t_j(w);w)}
           {\dot\rho(t_j(w);w)}.
\]
Hence $t_1'>0$ and $t_2'<0$ by Lemma~\ref{technical}.

Put $z_j(w):=z(t_j(w);w)$. Using Lemma~\ref{technical} and
$\dot\rho^2+w^2f(\bar r)^2=1$,
\[
    z_j'(w)
    =
    -\frac{\partial_w\rho(t_j(w);w)}
           {w\,\dot\rho(t_j(w);w)}.
\]
Thus $z_1$ is strictly increasing and $z_2$ strictly decreasing.
Moreover, $z_1(w)\to0$ as $w\downarrow0$, while the two branches meet
continuously at $w=w_*$.

To determine the other endpoint of the returning branch, define the
vertical displacement at the first return by
\[
    Z(w)
    :=
    z(P(w);w)
    =
    2w\int_0^{r_*(w)}
    \frac{f(r)^2}{\sqrt{1-w^2f(r)^2}}\,\dd r,
    \qquad w>0.
\]
Since $z_2(w)=Z(w)-z_1(w)$, it is enough to show
$Z(w)\to+\infty$ as $w\downarrow0$. After the change of
variables $s=wf(r)$,
\[
    \frac{Z(w)}2
    =
    \int_0^1
    \frac{f(f^{-1}(s/w))^2}
         {f'(f^{-1}(s/w))\sqrt{1-s^2}}\,\dd s.
\]
Now $f^2/f'=rf/\eta$, where $\eta=rf'/f$ is positive and
nonincreasing. Hence $f(r)^2/f'(r)\to+\infty$ as $r\to\infty$, and
Fatou's lemma gives $Z(w)\to+\infty$. Consequently
$z_2(w)\to+\infty$ as $w\downarrow0$.

It follows that for every $\bar z>0$ there is exactly one
$w\in(0,w_*]$ and exactly one $t\in(0,P(w))$ such that
$(\rho(t;w),z(t;w))=(\bar r,\bar z)$. Reflection gives the same
conclusion for $\bar z<0$, while $\bar z=0$ is reached uniquely by the
planar straight line $w=0$. Thus every off-axis point is reached by
at most one normal geodesic before its first return.

On the other hand, completeness of $(\R^3,d_{CC})$ and the
Hopf--Rinow theorem for locally compact length spaces give a length
minimizer from $q_0$ to every point. By
Proposition~\ref{prop:Hamilton} such a nonconstant minimizer is normal,
and Lemma~\ref{not minimizing past} shows that a minimizer ending
off $\Sigma$ must arrive strictly before its first return. The
preceding uniqueness therefore shows that every off-axis point is
joined to $q_0$ by a unique minimizing geodesic.

It remains to determine the first-return points on $\Sigma$. We already
know that $Z(w)\to+\infty$ as $w\downarrow0$. We claim that
$Z(w)\to0$ as $w\to+\infty$. Writing $a=r_*(w)$ and
$R_a(s)=f(as)/f(a)$, so that $w=1/f(a)$,
\[
    Z(w)
    =
    2af(a)
    \int_0^1
    \frac{R_a(s)^2}{\sqrt{1-R_a(s)^2}}\,\dd s.
\]
Fix $a_0>0$. Since $\eta$ is nonincreasing, for $0<a\le a_0$ one has
$R_a(s)\le s^{\eta(a_0)}$. Hence the integral above is uniformly
bounded by an integrable function independent of $a$, while
$af(a)\to0$ as $a\downarrow0$. Thus $Z(w)\to0$ as
$w\to+\infty$. The same representation shows that $Z$ is
continuous, so every point $(0,0,z)$ with $z>0$ occurs as a
first-return point; by reflection the same holds for $z<0$.

Finally, fix $w\neq0$. For every $t<P(w)$ the segment
$\gamma|_{[0,t]}$ is minimizing by the preceding uniqueness and
existence argument. Hence, by continuity of $d_{CC}$,
\[
    d_{CC}(q_0,\gamma(P(w)))
    =
    \lim_{t\uparrow P(w)}
    d_{CC}(q_0,\gamma(t))
    =
    P(w),
\]
so $\gamma|_{[0,P(w)]}$ is minimizing as well. At time $P(w)$ the
endpoint is independent of the initial planar direction, and hence
infinitely many minimizing geodesics meet there. By
Lemma~\ref{not minimizing past}, none remains minimizing beyond this
time. Therefore
\[
    t_{\operatorname{cut}}(\gamma)
    =
    \mathbf P(1/2,0,w).
\]
Rescaling recovers
$t_{\operatorname{cut}}(\gamma)=\mathbf P(E,0,w)$ for arbitrary
energy. Since every point of $\Sigma\setminus\{q_0\}$ occurs as a
first-return endpoint and no off-axis point is a cut point,
\[
    \operatorname{Cut}(q_0)=\Sigma\setminus\{q_0\}.
\]
\end{proof}

\begin{figure}
    \centering
    \includegraphics[width=0.5\linewidth]{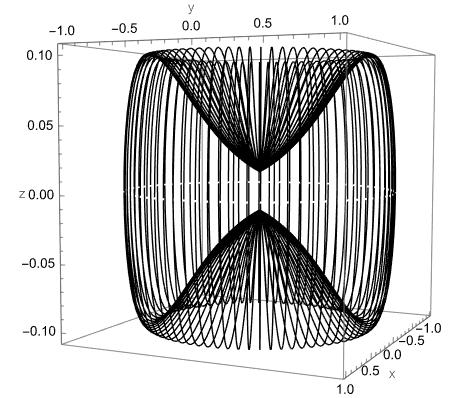}
    \caption{Unit Ball $B_{CC}(q_0,1)$ for the singular point $q_0=(0,0,0)$ in the radial Grushin structure with $f(r)=\log(r+1)^\beta r^\alpha$, and $\alpha=1,\beta=2$.}
    \label{fig:unit ball}
\end{figure}
\subsection{A Radial-Phase Maxwell Symmetry at Riemannian Points}

For general $f\in\mathfrak F$, we do not attempt to determine the
complete optimal synthesis from Riemannian points. Nevertheless, every
nonconstant radially periodic geodesic carries a natural \emph{radial-phase
symmetry} which produces a Maxwell time. In the homogeneous case $f(r)=r^\alpha$, this
will turn out to give the true cut time.

Let $q_0\notin\Sigma$. Recall that the momentum coordinates
$(L,K,w)$ on $T^*_{q_0}\R^3$ are given by
\[
    L=x_0u_0+y_0v_0,
    \qquad
    K=x_0v_0-y_0u_0,
    \qquad
    w=w_0.
\]
The transformation
\[
    (L,K,w)\longmapsto(-L,K,w)
\]
reverses the initial radial momentum while preserving the energy,
angular momentum, and vertical momentum.
For fixed $E$, $K$, and $w$, the radial equation is the same for the two
initial covectors; the signs $\pm L$ simply launch the same periodic
radial trajectory with opposite initial \emph{phase}. By the radial phase we mean the position of the initial data within one cycle of the periodic radial motion, specified by the initial radius and the direction of radial motion. Then, since
\[
    \dot\theta=\frac{K}{r^2},
    \qquad
    \dot z=wf(r)^2
\]
depend on the radial trajectory only through $r$, the two phase
partners accumulate the same angular and vertical displacement over
one complete radial period.  This is the radial-phase Maxwell
mechanism used throughout the homogeneous synthesis.

\begin{theorem}\label{Geodesic and Cut time Theorem}
Let $f\in \mathfrak{F}_R^{\mathbf{P}}$ and $q_0\notin\Sigma$. Let $\gamma$ be a geodesic whose initial covector has momentum coordinates $(L,K,w)$, where $wL\neq0$, which is radially periodic by Proposition~\ref{prop:radial-periodicity}.
Let $\widehat\gamma$ be the geodesic with initial momentum coordinates
$(-L,K,w)$. Then, writing
$P=\mathbf P_{q_0}(L,K,w)$,
we have
$\gamma(P)=\widehat\gamma(P)$.
Consequently,
\[
    t_{\operatorname{cut}}(\gamma)\leq
    \mathbf P_{q_0}(L,K,w).
\]
\end{theorem}

\begin{proof}
Let $\gamma$ and $\widehat\gamma$ have initial covectors with momentum coordinates
$(L,K,w)$ and $(-L,K,w)$.

Suppose first that $K\neq0$. By uniqueness for the radial equation,
$\widehat r(t)=r(-t)$. Since $r$ is $P$-periodic, we have
$\widehat r(t)=r(P-t)$. In particular,
$\widehat r(P)=r(P)=r_0$. For any continuous function $\varphi$
on the interval of radial motion, the substitution $s=P-t$ gives
\[
    \int_0^P\varphi(\widehat r(t))\,\dd t
    =
    \int_0^P\varphi(r(P-t))\,\dd t
    =
    \int_0^P\varphi(r(s))\,\dd s.
\]
Applying this identity to $\varphi(r)=r^{-2}$ and
$\varphi(r)=f(r)^2$, and recalling that both geodesics have the same
initial point and the same values of $K$ and $w$, we obtain
\[
    \widehat\theta(P)=\theta(P),
    \qquad
    \widehat z(P)=z(P).
\]
Thus $\widehat\gamma(P)=\gamma(P)$.

If $K=0$, let $\rho$ and $\widehat\rho$ be the signed radial
coordinates, oriented along the planar position of $q_0$. The
signed radial equation has an even potential, and its solutions
satisfy
\[
    \widehat\rho(t)=\rho(-t),
    \qquad
    \rho(t+P)=-\rho(t).
\]
Consequently $\widehat\rho(P)=\rho(-P)
    =-\rho(0)=\rho(P)$. Moreover, since $F(\rho)=f(|\rho|)^2$ is even,
\[
    F(\widehat\rho(t))
    =F(\rho(-t))
    =F(\rho(P-t)).
\]
The same change of variables therefore gives
\[
    \widehat z(P)-z_0
    =w\int_0^P F(\widehat\rho(t))\,\dd t
    =w\int_0^P F(\rho(t))\,\dd t
    =z(P)-z_0.
\]
Hence $\widehat\gamma(P)=\gamma(P)$ also in this case.

Since $L\neq0$, the initial covectors are distinct, so $P$ is a
Maxwell time. Furthermore, the two geodesics arrive at their common
endpoint with opposite radial velocities, which are nonzero.
If $\gamma$ minimized beyond $P$, replacing its initial segment
by $\widehat\gamma|_{[0,P]}$ would produce a minimizing curve with
a corner at this Riemannian endpoint. This is impossible, and
therefore $t_{\mathrm{cut}}(\gamma)\leq P$. 
\end{proof}
\begin{remark}\label{Angular Reflection Remark}
There is a second Maxwell symmetry obtained by fixing $(L,w)$ and
replacing $K$ by $-K$. For $K\neq0$ the corresponding geodesics have
identical radial components and meet when their angular separation
first reaches $2\pi$. This remains a valid upper bound for the cut
time, but for the homogeneous profiles $f(r)=r^\alpha$, $\alpha>1$, considered in the sequel, the
radial-phase Maxwell time occurs first. We remark that in the integrable case $f(r)=r$, this Maxwell symmetry and the Maxwell symmetry described by Theorem~\ref{Geodesic and Cut time Theorem} actually coincide and that this is the true cut time for Riemannian geodesics \cite{albert2026measurecontractionpropertygrushin}. 
\end{remark}
\section{Optimal Synthesis for the Power Law Family}

We now specialize to the homogeneous family
 $f(r)=r^\alpha$, $\alpha>1$
and determine the optimal synthesis from Riemannian points. We normalize until further notice
the base point to $q_0=(1,0,0)$ and work at unit speed $2E=1$. Write $P_\alpha(K,w):=\mathbf P(1/2,K,w)$ for the unit-speed radial period, with the continuous extension at circular orbits from Theorem~\ref{lem:radial-period-continuity}. Our goal is to
verify the hypotheses of the Extended Hadamard technique,
Theorem~\ref{extended Hadamard}; the main difficulty is the
nonconjugacy condition~\ref{loc diffeo EHT}.

\textbf{Outline of the proof:}
The argument has four steps. First, homogeneity reduces the radial
dynamics to a one-parameter family of oscillators. The associated
angular advance over one radial period is the \emph{apsidal angle}
$\omega_\alpha$, whose strict monotonicity controls both conjugacy and
the geometry of the full-period endpoint map.

Second, we study the differential of the exponential map and prove that
no geodesic is conjugate before one complete radial period. 

Third, we compute the full-period endpoints. Their image consists of
two tongue-shaped subsets of the cylinder $\{r=1\}$, which form the
candidate cut locus. We then show that every nontrivial return to
$\{r=1\}$ occurring before $P_\alpha$ lies outside this set.

Finally, we show that the exponential map before $P_\alpha$ is a
proper local diffeomorphism onto the complement of the candidate cut
locus. This complement is simply connected, so the Extended Hadamard
theorem makes the exponential map a global diffeomorphism there.
Together with the radial-phase Maxwell symmetry, this proves that $t_{\operatorname{cut}}=P_\alpha$,
and via Theorem~\ref{singular point optimal synth} completes the optimal synthesis for this model. 
\section{Conjugate-Time Analysis}\label{sec:generic-conj}

We establish the conjugate-time estimate required by \ref{loc diffeo EHT}. The generic covectors and the strata $K=0$ and $L=0$ are treated separately, using the apsidal-angle monotonicity and the endpoint Jacobian.

Related methods have been employed in the study of conjugate and cut loci on surfaces of revolution. In particular, the analysis of Bonnard et al. \cite{BonnardCaillauSinclairTanaka2009,BonnardCaillauJanin2013} is similar. 

\subsection{Scaled radial dynamics}

Assume first that $w\neq0$ and set
\[
    s:=|w|^{1/\alpha},\qquad
    \kappa:=|K|s,\qquad
    Q_\kappa(u):=1-\frac{\kappa^2}{u^2}-u^{2\alpha}.
\]
Then $\dot r^2=Q_\kappa(sr)$.  There is a critical value
$\kappa_c=\sqrt{\alpha}\,(\alpha+1)^{-(\alpha+1)/(2\alpha)}$ such that,
for $0<\kappa<\kappa_c$, the function $Q_\kappa$ has exactly two
positive roots $a(\kappa)<b(\kappa)$.  These are the inner and outer
scaled turning radii.  At $\kappa=\kappa_c$ they coalesce at
$s_c=(\alpha+1)^{-1/(2\alpha)}$, which corresponds to the circular
unit-speed geodesic $r\equiv1$.

Define
\begin{equation}\label{p and omega def}
    p_\alpha(\kappa)
    :=2\int_{a(\kappa)}^{b(\kappa)}
       \frac{\dd u}{\sqrt{Q_\kappa(u)}},
    \qquad
    \omega_\alpha(\kappa)
    :=2\kappa\int_{a(\kappa)}^{b(\kappa)}
       \frac{\dd u}{u^2\sqrt{Q_\kappa(u)}}.
\end{equation}
Thus $p_\alpha$ is the scaled radial period and $\omega_\alpha$ is the
\emph{absolute angular advance} over one complete radial period.  In
physical variables,
\begin{equation}\label{eq:scaled-period-angle}
    P_\alpha(K,w)=s^{-1}p_\alpha(\kappa),
    \qquad
    \theta(P_\alpha)=\operatorname{sgn}(K)\omega_\alpha(\kappa).
\end{equation}

The monotonicity of $\omega_\alpha$ will be the crux of our analysis, and it follows from the classical theory
of apsidal angles.

\subsection{Apsidal angles}

A planar central-force problem has the form
$\ddot \xi=-V'(r)\xi/r$, where $r=|\xi|$.  The angular momentum
$\ell=r^2\dot\theta$ is conserved, and the radial motion is governed by
the effective potential $V(r)+\ell^2/(2r^2)$.  If a bounded
noncircular orbit has consecutive turning radii $r_-<r_+$, its
\emph{apsidal angle} is the angular displacement from one radial turn
to the next,
\[
    \Theta(h,\ell)
    =\int_{r_-}^{r_+}
      \frac{\ell\,\dd r}
      {r^2\sqrt{2(h-V(r))-\ell^2/r^2}}.
\]
Hence $2\Theta$ is the angular advance over one complete radial
oscillation.  Standard references for planar central-force mechanics
are \cite[Ch.~3]{GoldsteinPooleSafko} and
\cite[\S14]{LandauLifshitzMechanics}.

Rojas \cite{Rojas2018} proves monotonicity of the apsidal angle for
power-law forces using a period criterion of Schaaf
\cite{SchaafPeriods}.  In the form needed here, Schaaf's criterion says
that the period $T(h)$ of the oscillator $\ddot x+U'(x)=0$ is strictly
increasing in its energy $h$ if
\[
    5(U^{(3)})^2-3U''U^{(4)}>0
    \quad\text{where }U''>0,
    \qquad
    U'U^{(3)}<0
    \quad\text{where }U''=0,
\]
and is strictly decreasing if the first expression is negative wherever
$U''\ge0$.

For the power-law potential $V(r)=-r^{-\beta}/\beta$, Rojas makes the
Clairaut substitution $\rho=\ell/r$.  The apsidal-angle integral then
becomes, up to the usual factor of two, the period integral of the
one-dimensional oscillator with potential
$W_{\ell,\beta}(\rho)=\rho^2/2-\ell^{-\beta}\rho^\beta/\beta$.
A second substitution, $u=\ell^\beta\rho^{2-\beta}$, reduces Schaaf's
conditions to the signs of affine functions of $u$.  This gives
$\partial_h\Theta_\beta<0$ for $1/2<\beta<1$ and
$\partial_h\Theta_\beta>0$ for $1<\beta<2$.  The remaining exponents
are obtained from the power-law duality
$(2-\beta)(2-\widehat\beta)=4$, together with the corresponding
rescaling of the apsidal angle; see \cite{GrantRosner}.  In particular,
Rojas obtains $\partial_h\Theta_\beta>0$ for $\beta<-2$.

\begin{proposition}[Strict apsidal twist]\label{prop:apsidal}
Let $\alpha>1$.  Then $\omega_\alpha'(\kappa)<0$ for
$0<\kappa<\kappa_c$, and
\[
    \frac{2\pi}{\sqrt{2(\alpha+1)}}
    <\omega_\alpha(\kappa)<\pi.
\]
\end{proposition}

\begin{proof}
The planar component of a geodesic satisfies
$\ddot \xi=-\alpha w^2r^{2\alpha-2}\xi$.  In the notation above this is a
power-law central force with exponent $\beta=-2\alpha<-2$, so Rojas'
theorem gives $\partial_E\Theta(E,K,w)>0$ at fixed $K$ and $w$.
Restoring the energy,
\[
    \kappa(E)
    =|K||w|^{1/\alpha}(2E)^{-(\alpha+1)/(2\alpha)},
    \qquad
    2\Theta(E,K,w)=\omega_\alpha(\kappa(E)).
\]
Since $\kappa'(E)<0$, the first claim follows.  Finally,
$\omega_\alpha(\kappa)\to\pi$ as $\kappa\downarrow0$, while the
small-oscillation limit at the circular orbit is
$2\pi/\sqrt{2(\alpha+1)}$.  Strict monotonicity gives the stated
bounds.
\end{proof}

The strict twist is the only external input in the conjugacy argument
below. The
rest of the generic argument comes directly from the Hamiltonian
system and its variations, and can be applied to other $f\in \mathfrak{F}_R^{\mathbf{P}}$, although we do not pursue this here. 

\subsection{Generic covectors}

We first assume $KLw\neq0$. We call covectors with these momentum coordinates \emph{generic}. On the unit energy shell, $(K,w)$ are then
local coordinates once the sign of
$L=\pm\sqrt{1-K^2-w^2}$ is fixed.  By reflection symmetry it is enough
to treat $K,w>0$.  Since $K\neq0$, the radial coordinate stays positive
and cylindrical coordinates are nonsingular along the geodesic.

Recall that a \emph{radial branch} is a maximal open interval on which
$\dot r\neq0$ and that a time at which $\dot r=0$ is a \emph{radial turn};
for a noncircular orbit the turns considered below are simple, meaning
$\ddot r\neq0$ there.

On a radial branch, consider the ordinary endpoint map Jacobian
$J=\partial(r,\theta,z)/\partial(t,K,w)$.  Subtracting
$(r_K/\dot r)$ times the first column from the second, and
$(r_w/\dot r)$ times the first column from the third, leaves the
determinant unchanged.  We therefore define $M$ to be the lower
$2\times2$ block after this elimination:
\[
M=
\begin{pmatrix}
\theta_K-\dfrac{\dot\theta}{\dot r}r_K&
\theta_w-\dfrac{\dot\theta}{\dot r}r_w\\[2mm]
z_K-\dfrac{\dot z}{\dot r}r_K&
z_w-\dfrac{\dot z}{\dot r}r_w
\end{pmatrix}.
\]
Thus $M$ records the two endpoint variations after their radial
components have been removed.

\begin{lemma}\label{lem:branchwise-M}
On every radial branch,
\begin{equation}\label{eq:J-M}
    \det\frac{\partial(r,\theta,z)}{\partial(t,K,w)}
    =\dot r\,\det M.
\end{equation}
Moreover, $M$ is symmetric and
\begin{equation}\label{eq:M-prime}
    \dot M=
    \frac1{\dot r^2}
    \begin{pmatrix}
        r^{-2}(1-w^2r^{2\alpha}) & Kw r^{2\alpha-2}\\
        Kw r^{2\alpha-2} & r^{2\alpha}(1-K^2/r^2)
    \end{pmatrix}
    \succ0.
\end{equation}
Consequently the two eigenvalues of $M$ are strictly increasing on
each radial branch.
\end{lemma}

\begin{proof}
The factorization \eqref{eq:J-M} is exactly the column reduction just
described.  To compute $\dot M$, use $r$ as the independent variable
on the branch.  Since
$\dot r^2=1-K^2/r^2-w^2r^{2\alpha}$,
$\dd\theta/\dd r=K/(r^2\dot r)$, and
$\dd z/\dd r=wr^{2\alpha}/\dot r$, differentiating with respect to
$(K,w)$ at fixed $r$ gives \eqref{eq:M-prime}.  If $B$ denotes its
numerator, then
$\det B=r^{2\alpha-2}\dot r^2>0$, and the upper-left entry of $B$ is
positive.  Hence $\dot M$ is positive definite.

It remains to verify that $M$ itself is symmetric on every branch.
The radial variational equations give
$\ddot r\,r_w=\dot r\,\dot r_w+wr^{2\alpha}$ and
$\ddot r\,r_K=\dot r\,\dot r_K+K/r^2$.  Differentiation of \eqref{Radial ODE full}
shows that
$r_K\dot r_w-r_w\dot r_K-\theta_w+z_K$ is constant; it vanishes at
$t=0$ because the initial point is fixed.  Combining these three
identities yields
\[
    \theta_w-\frac{\dot\theta}{\dot r}r_w
    =
    z_K-\frac{\dot z}{\dot r}r_K,
\]
which is precisely $M_{12}=M_{21}$.  Finally, if $t_1<t_2$ lie on one
branch, then $M(t_2)-M(t_1)$ is the integral of the positive-definite
matrix $\dot M$.  The min--max characterization of the eigenvalues
gives the claimed strict monotonicity.
\end{proof}

The preceding lemma is local along a radial branch.  The next two
results provide the global information needed to pass from the first
branch to the full radial return.

\begin{lemma}\label{lem:return-M-negative}
For every geodesic with initial covector in the generic regime,
$M(P_\alpha)$ is strictly negative definite.
\end{lemma}

\begin{proof}
Again take $K,w>0$.  From the radial equation,
$\frac{\dd}{\dd t}(r\dot r)=1-(\alpha+1)w\dot z$.  Since
$r(0)=1$, $\dot r(0)=L$, and $z(0)=0$, integration gives the \emph{virial
identity} (which will have later applications)
\begin{equation}\label{eq:virial}
    (\alpha+1)wz(t)=t-r(t)\dot r(t)+L.
\end{equation}
At the full radial period, $r(P_\alpha)=1$ and
$\dot r(P_\alpha)=L$, so
\begin{equation}\label{eq:period-z}
    Z_\alpha(K,w):=z(P_\alpha)
    =\frac{P_\alpha(K,w)}{(\alpha+1)w}.
\end{equation}

Because $L\neq0$, the full radial period lies in the interior of a radial
branch, so $M(P_\alpha)$ is the differential of $(\omega_\alpha(\kappa),Z_\alpha)$ with respect to $(K,w)$ with the terminal radius held constant. Then,
using \eqref{eq:scaled-period-angle}, \eqref{eq:period-z}, and
Proposition~\ref{prop:apsidal}, one obtains by the chain rule
\begin{equation}\label{eq:return-Jacobian}
    \partial_K\omega_\alpha(\kappa)
    =w^{1/\alpha}\omega_\alpha'(\kappa)<0,
    \qquad
    \det M(P_\alpha)
    =-\frac{p_\alpha(\kappa)}{\alpha w^2}
      \omega_\alpha'(\kappa)>0.
\end{equation}
The matrix $M(P_\alpha)$ is symmetric by Lemma~\ref{lem:branchwise-M}; a symmetric
$2\times2$ matrix with negative upper-left entry and positive
determinant is negative definite.
\end{proof}

At a radial turn the matrix $M$ is undefined because $r$ is no longer a
valid endpoint coordinate. However, its asymptotics at these radial turns are both tractable and extremely useful. 

\begin{lemma}\label{lem:turning-Jacobian}
Let $t_*$ be a radial turn of a geodesic in the generic regime.
If $M$ is definite on one of the two adjacent radial branches, then
\[
    \det\frac{\partial(r,\theta,z)}
              {\partial(t,K,w)}(t_*)\neq0.
\]
\end{lemma}

\begin{proof}
Put $r_*=r(t_*)$ and $c=\ddot r(t_*)\neq0$.  Writing
$\tau=t-t_*$, the radial symmetry about a simple turn gives
$\dot r=c\tau+O(\tau^3)$ and
$r=r_*+\frac12c\tau^2+O(\tau^4)$.  Let $B(r)$ be the numerator in
\eqref{eq:M-prime}.  At the turn,
$K^2/r_*^2+w^2r_*^{2\alpha}=1$, so
$B(r_*)=b_*b_*^T$ with
$b_*=(K/r_*^2,wr_*^{2\alpha})^T\neq0$.

Let $u_*=b_*/|b_*|$ and choose a unit vector $e_*$ orthogonal to
$u_*$.  Since $B(r(t))=B(r_*)+O(\tau^2)$, the quantities
$e_*^TM e_*$ and $u_*^TM e_*$ have finite one-sided limits, whereas
\[
    u_*^TM(t)u_*
    =-\frac{|b_*|^2}{c^2\tau}+O(1).
\]
Write $m_-$ and $m_+$ for the one-sided limits of $e_*^TM e_*$.  If
$M$ is definite on the branch preceding the turn, the displayed
asymptotic forces it to be positive definite there.  Since
$\dot M\succ0$, $e_*^TM e_*$ is strictly increasing, and therefore
$m_->0$.  If the definite branch follows the turn, the same argument
shows that it is negative definite there and $m_+<0$.  Thus the
one-sided limit corresponding to a definite adjacent branch is
nonzero.

In the basis $\{u_*,e_*\}$ the bounded off-diagonal term does not
contribute to the leading order of the determinant, and hence
\[
    \lim_{t\to t_*^\pm}\dot r(t)\det M(t)
    =-\frac{|b_*|^2}{c}\,m_\pm.
\]
By \eqref{eq:J-M}, this is the one-sided limit of the ordinary endpoint
Jacobian.  That Jacobian is continuous at $t_*$, so the two limits
agree; in particular $m_-=m_+$.  Since one of them is nonzero,
\begin{equation}\label{eq:turning-Jacobian-limit}
    \det\frac{\partial(r,\theta,z)}
              {\partial(t,K,w)}(t_*)
    =-\frac{|b_*|^2}{c}\,m_*\neq0.
\end{equation}
\end{proof}

\begin{proposition}[Generic non-conjugacy]\label{prop:generic-conj}
If $KLw\neq0$, then
\[
    t_{\operatorname{con}}(L,K,w)>P_\alpha(K,w).
\]
\end{proposition}

\begin{proof}
Suppose first that $L>0$, and let $0<t_1<t_2<P_\alpha$ be the two
radial turns.  On the first branch, $M(0)=0$ and
Lemma~\ref{lem:branchwise-M} gives $M\succ0$.  On the final branch,
Lemma~\ref{lem:return-M-negative} and the same monotonicity give
$M(t)\prec M(P_\alpha)\prec0$.  Thus the endpoint Jacobian is nonzero
on the first and final branches, and
Lemma~\ref{lem:turning-Jacobian} shows that it is also nonzero at
$t_1$ and $t_2$.

By continuity, the endpoint determinant has the same nonzero sign on
the middle branch near both turns.  Since $\dot r<0$ there,
\eqref{eq:J-M} implies that $M$ is indefinite near both ends of the
middle branch.  Its larger eigenvalue is already positive just after
$t_1$ and remains positive because both eigenvalues increase.  Its
smaller eigenvalue is negative just before $t_2$ and, by the same
monotonicity, must have been negative throughout the branch.
Therefore $M$ is nonsingular on the entire middle branch.  The
endpoint Jacobian is consequently nonzero for
$0<t\le P_\alpha$.

Because $r(t)>0$, passing from cylindrical to Cartesian endpoint
coordinates multiplies the determinant by the nonzero factor $r(t)$,
so no conjugate time occurs before or at $P_\alpha$.  The case $L<0$
is the same argument with the radial branches traversed in the opposite
order.
\end{proof}

The remaining strata are treated separately because either cylindrical
coordinates degenerate at the axis or $(K,w)$ cease to be coordinates
on the unit energy shell.

\subsection{Zero angular momentum}

Assume $K=0$.  After rotating around the $z$-axis, the geodesic lies
in the plane $y=0$.  Write $x$ for the signed radial coordinate, so
$x(0)=1$, $\dot x(0)=L$, and
$\dot x^2+w^2|x|^{2\alpha}=1$.  The planar endpoint map is the one for
the $\alpha$-Grushin plane, which is well studied.  Its conjugate-time calculation is known
\cite{Borza2022}: it is nonsingular for $0<t<P_\alpha$, remains
nonsingular at $P_\alpha$ when $L\neq0$, and is singular there when
$L=0$.  It remains only to check that the variation transverse to the
plane does not introduce an earlier conjugate time.

Let $V=\partial_Ky|_{K=0}$.  Since
$K=x\dot y-y\dot x$ at the initial point, $V(0)=0$ and
$\dot V(0)=1$.  Linearizing in the transverse direction gives
\begin{equation}\label{eq:transverse-Kzero}
    \ddot V+\alpha w^2|x|^{2\alpha-2}V=0.
\end{equation}
The signed radial function $x$ solves the same equation, and their
Wronskian is $xV'-x'V\equiv1$.  Combining this identity with the energy
equation gives
\[
    \left(\frac{V+x'}{x}\right)'
    =-(\alpha-1)w^2|x|^{2\alpha-2}.
\]
The quotient extends continuously across the unique zero of $x$ in
$(0,P_\alpha)$.  Using
$x(P_\alpha)=-1$ and $x'(P_\alpha)=-L$ therefore yields
\begin{equation}\label{eq:V-return}
    V(P_\alpha)
    =(\alpha-1)w^2
      \int_0^{P_\alpha}|x(t)|^{2\alpha-2}\,\dd t>0.
\end{equation}

Now $V>0$ for small positive time.  If $V$ vanished somewhere in
$(0,P_\alpha)$, we may take the smallest zero, which would evidently be a simple zero by uniqueness in the ODE driving $V$, then \eqref{eq:V-return} would force a second zero
before $P_\alpha$.  Sturm separation (See e.g \cite[Chapter 4]{Hartman1964ODE}) would then place two distinct
zeros of $x$ between consecutive zeros of $V$, contradicting the fact
that $x$ has only one zero on this interval.  Hence $V(t)\neq0$ for
$0<t\le P_\alpha$.

At $K=0$, the $t$- and $w$-variations remain in the plane $y=0$,
whereas the $K$-variation is transverse to it. The  ordinary endpoint map
determinant therefore factors, up to sign, as $V(t)$ times the planar
endpoint determinant.  It follows that
$t_{\operatorname{con}}>P_\alpha$ whenever $K=0$ and $L\neq0$.

If $K=L=0$, then $|w|=1$.  Using $(K,L)$ as coordinates on the unit
energy shell, the planar calculation gives
$\partial_Lx=-\dot x/\alpha$ and
$\det\partial(x,z)/\partial(t,L)=(w/\alpha)\dot x$.  Its first positive
zero is $P_\alpha$, while the transverse $K$-variation is nonzero
there by \eqref{eq:V-return}.  Hence
$t_{\operatorname{con}}=P_\alpha$ in this case.

\subsection{Zero radial momentum}

Assume now that $L=0$ and $Kw\neq0$, so $K^2+w^2=1$.  Use local
coordinates $(\beta,L)$ on the unit energy shell (with $\beta$ unrelated to the force exponent above), where
$(K_\beta,w_\beta)=(w,-K)$ at the covector under consideration.
Set $c_0=\ddot r(0)=K^2-\alpha w^2$ and first suppose $c_0\neq0$;
the \emph{circular case} $c_0=0$ will be treated at the end.

The $L$-variation is especially simple.  Differentiating the radial
equation with respect to $L$, and comparing with the time derivative
of that same equation, gives
\begin{equation}\label{eq:rL-Lzero}
    r_L=\frac{\dot r}{c_0}.
\end{equation}
After differentiating $\dot\theta=K/r^2$ and
$\dot z=wr^{2\alpha}$ and using \eqref{eq:rL-Lzero}, one finds that,
when the terminal radius is held fixed,
\begin{equation}\label{eq:L-fixed-r}
    \theta_L^r=-\frac{K}{c_0},
    \qquad
    z_L^r=-\frac{w}{c_0}.
\end{equation}
Here, for example,
$\theta_L^r=\theta_L-(r_L/\dot r)\dot\theta$ on a radial branch.

Define in the same way the fixed-radius $\beta$-variations
$\theta_\beta^r,z_\beta^r$, and set
$\chi=Kz_\beta^r-w\theta_\beta^r$.  Eliminating the radial component
from the endpoint Jacobian and using \eqref{eq:L-fixed-r} gives
\begin{equation}\label{eq:Lzero-Jacobian}
    \det\frac{\partial(r,\theta,z)}{\partial(t,\beta,L)}
    =\frac{\dot r}{c_0}\chi.
\end{equation}
Differentiating the radial formulae for
$\dd\theta/\dd r$ and $\dd z/\dd r$ gives the key monotonicity:
\begin{equation}\label{eq:chi-prime}
    \chi'(t)=-
    \left[
        \frac{w^2}{r^2}
        +K^2r^{2\alpha}
        +\frac{K^2w^2(r^{-2}-r^{2\alpha})^2}{\dot r^2}
    \right]<0.
\end{equation}
At the initial turn, the last term tends to zero, while the first two
terms tend to $1$.  Thus $\chi(t)\to0$ and $\chi<0$ on the first radial
branch.

Let $t_*$ be the unique interior radial turn and put
$r_*=r(t_*)$, $c=\ddot r(t_*)$.  Since the geodesic is not radially constant,
$r_*\neq1$.  With
$D_*=K^2w^2(r_*^{-2}-r_*^{2\alpha})^2>0$, the singular term in
\eqref{eq:chi-prime} gives
$\chi(t)=D_*/(c^2(t-t_*))+O(1)$.  Hence $\chi$ tends to $-\infty$
from the first branch and to $+\infty$ from the returning branch.
Together with \eqref{eq:Lzero-Jacobian}, this gives
\begin{equation}\label{eq:Lzero-turn-limit}
    \lim_{t\to t_*}
    \det\frac{\partial(r,\theta,z)}{\partial(t,\beta,L)}
    =\frac{D_*}{c_0c}\neq0.
\end{equation}
Thus the interior radial turn is not conjugate.

It remains to determine the sign of $\chi$ at $P_\alpha$. There
the $\beta$-variation is differentiation of the endpoint cylindrical coordinates $\theta(P_\alpha)$ and $z(P_\alpha)$ in
the tangent direction $v=(w,-K)$.  Since
$\theta(P_\alpha)=\omega_\alpha(\kappa)$,
$z(P_\alpha)=p_\alpha(\kappa)/((\alpha+1)w^{1+1/\alpha})$, and
$\kappa=Kw^{1/\alpha}$, we have
$D_v\kappa=w^{1/\alpha-1}(w^2-K^2/\alpha)$.  Comparing the
off-diagonal entries of the generic regime full period matrix $M(P_\alpha)$ gives the
identity
$p_\alpha'(\kappa)=((\alpha+1)/\alpha)\kappa
\omega_\alpha'(\kappa)$.  Substitution yields
\begin{equation}\label{eq:chi-return}
    \chi(P_\alpha)
    =
    -\omega_\alpha'(\kappa)
      w^{1/\alpha-2}
      \left(w^2-\frac{K^2}{\alpha}\right)^2
    +\frac{K^2p_\alpha(\kappa)}
           {\alpha w^{2+1/\alpha}}>0.
\end{equation}
Indeed, $\omega_\alpha'<0$ by Proposition~\ref{prop:apsidal} and
$p_\alpha>0$.  Since $\chi$ decreases from $+\infty$ on the returning
branch to the positive value \eqref{eq:chi-return}, it never vanishes
there.  Equations \eqref{eq:Lzero-Jacobian} and
\eqref{eq:Lzero-turn-limit} therefore show that the endpoint
determinant is nonzero for $0<t<P_\alpha$.  As $t\uparrow P_\alpha$,
$\dot r\to0$ while $\chi\to\chi(P_\alpha)>0$, so the determinant tends
to zero.  Hence $t_{\operatorname{con}}=P_\alpha$ for every
noncircular geodesic with $L=0$ and $Kw\neq0$.

For the circular geodesic, $K^2=\alpha/(\alpha+1)$ and
$w^2=1/(\alpha+1)$.  Put $\nu=\sqrt{2\alpha}$.  Linearizing in the
coordinates $(\beta,L)$ and integrating the resulting constant
coefficient equations gives
\begin{equation}\label{eq:circular-Jacobian}
    \det\frac{\partial(r,\theta,z)}{\partial(t,\beta,L)}
    =\frac1{2\alpha}
      \bigl[\nu t\sin(\nu t)+4(\cos(\nu t)-1)\bigr].
\end{equation}
Writing $x=\nu t/2$, the bracket is a positive constant multiple of
$\sin x\,(x\cos x-2\sin x)$.  This does not vanish for
$0<x<\pi$: for $x\ge\pi/2$ the claim is immediate, while for
$0<x<\pi/2$ it follows from $\tan x>x$.  Its first positive zero is
therefore $x=\pi$, or $t=2\pi/\sqrt{2\alpha}=P_\alpha$.  Thus the
circular geodesic also satisfies $t_{\operatorname{con}}=P_\alpha$.

Finally, if $w=0$, the geodesic is a planar straight line.  A
short computation shows that it has no conjugate times.

We have the following non-conjugacy theorem as a result of the analysis in the previous three subsections. 

\begin{theorem}\label{thm:conjugate-comparison}
Let $\gamma$ be a unit-speed geodesic beginning at
$q_0=(1,0,0)$.  If $w=0$, then
$t_{\operatorname{con}}=+\infty$.  If $w\neq0$, then
\[
    P_\alpha(K,w)\le t_{\operatorname{con}}(L,K,w),
\]
where $P_\alpha$ is understood by continuous extension on the
degenerate strata.  More precisely,
\[
    t_{\operatorname{con}}(L,K,w)=P_\alpha(K,w)
    \quad\Longleftrightarrow\quad L=0,
\]
whereas $t_{\operatorname{con}}>P_\alpha$ whenever $L\neq0$.
\end{theorem}
Consequently, the upper bound $t_{\operatorname{cut}}(\gamma)\le P_\alpha$ holds for every unit-speed geodesic from $q_0=(1,0,0)$: it follows from Theorem~\ref{Geodesic and Cut time Theorem} when $wL\neq0$, from the conjugate-time bound in Theorem~\ref{thm:conjugate-comparison} when $w\neq0$ and $L=0$, and is immediate when $w=0$.

\section{Extended Hadamard Technique}\label{sec:extended hadamard}
We complete the homogeneous Riemannian synthesis by the Extended Hadamard argument. As in Section~\ref{sec:generic-conj}, take $q_0=(1,0,0)$ and $2E=1$; the general case follows by anisotropic homotheties. The full-period endpoint map gives a candidate cut locus in $\{r=1\}$, and we show that earlier nontrivial cylinder returns avoid it. Throughout this section, angular coordinates are understood modulo $2\pi$, with representatives in $[-\pi,\pi]$, and $\theta=\pi$ is the antipodal meridian. Boundaries in $\{r=1\}$ are relative to that cylinder unless otherwise specified.

Recall $s=|w|^{1/\alpha}$, $\kappa=|K|s$, and $Q_\kappa$ from Section~\ref{sec:generic-conj}.
With the normalization $r(0)=1$, the unit speed equation $2E=1$ is 
\begin{equation}\label{eq:Q-shell}
    L^2=Q_\kappa(s).
\end{equation}
For $w\neq0$, a unit covector exists precisely when $0\leq \kappa\leq \kappa_c$ and
$0<s$, $a(\kappa)\leq s\leq b(\kappa)$. The positive roots satisfy $a(\kappa)<b(\kappa)$ for $0<\kappa<\kappa_c$, with the endpoint conventions $a(0)=0$, $b(0)=1$, and $a(\kappa_c)=b(\kappa_c)=s_c$.
For interior points of this parameter region there are exactly two unit covectors, corresponding to
$L=\pm\sqrt{Q_\kappa(s)}$. See Figure \ref{parameter region}.

\begin{figure}[htbp]
    \centering
    \includegraphics[width=0.75\linewidth]{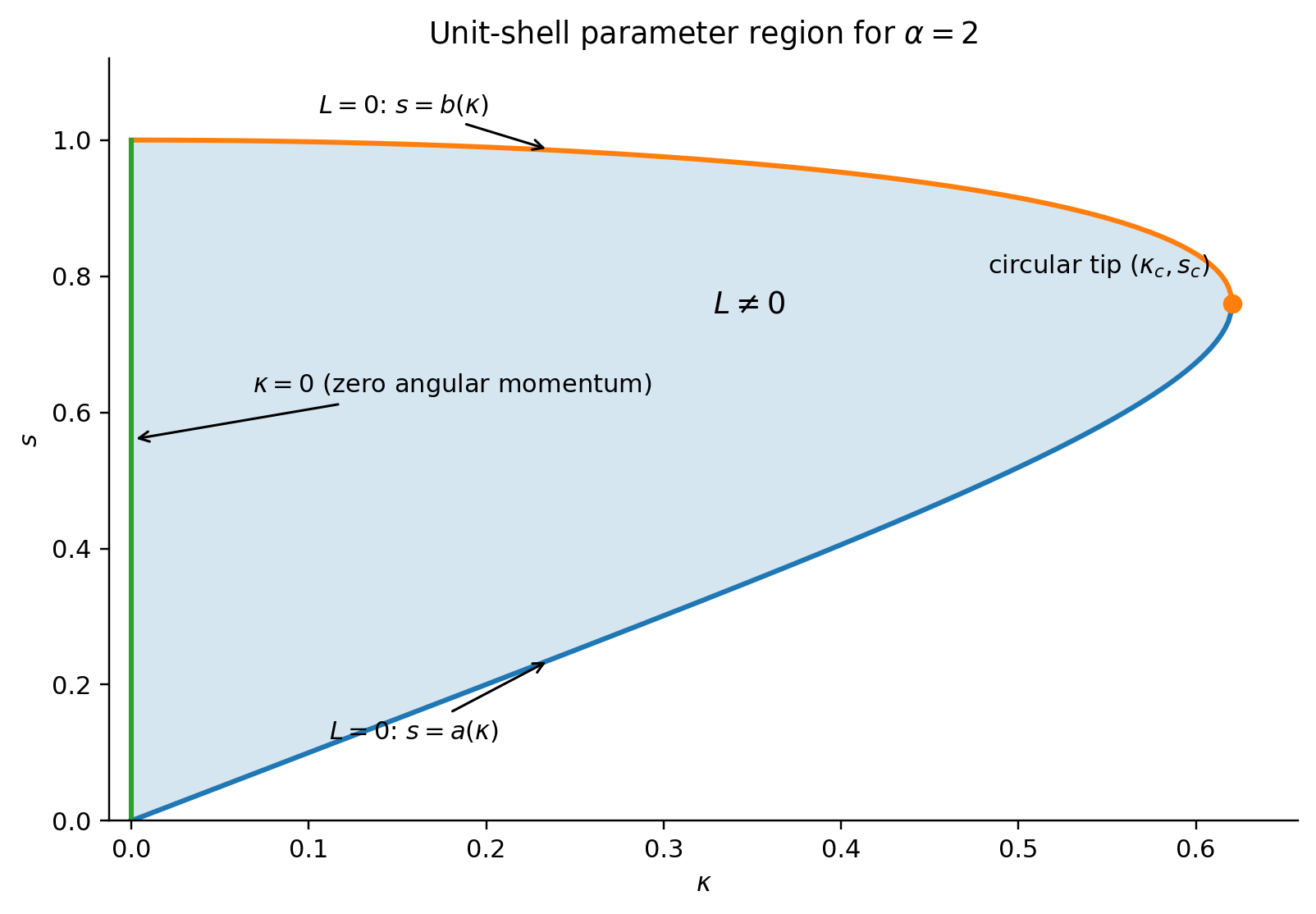}
    \caption{Parameter region in the $(\kappa,s)$-plane.}
    \label{parameter region}
\end{figure}
Theorem~\ref{Geodesic and Cut time Theorem} shows that the two covectors above have the same endpoint after one full radial period. We therefore begin by describing the image of the full-period endpoint map.

Let $E_1=\{\lambda\in T^*_{q_0}\R^3:H_f(q_0,\lambda)=1/2\}$ be the unit cotangent shell. The full-period endpoint map introduced above is
\[
    \mathcal R(\lambda):=\Exp_{q_0}(P_\alpha(\lambda)\lambda)=\gamma(P_\alpha(\lambda);\lambda),
    \qquad \lambda\in E_1\cap\{w\neq0\}.
\]
In cylindrical coordinates, the endpoint formula is
\begin{equation}\label{eq:cut-param}
    \mathcal R(\lambda)
    =
    \left(
        1,
        \operatorname{sgn}(K)\,\omega_\alpha(\kappa),
        \operatorname{sgn}(w)\,\frac{p_\alpha(\kappa)}{(\alpha+1)s^{\alpha+1}}
    \right).
\end{equation}
At $K=0$, the angular coordinate in \eqref{eq:cut-param} is understood to be $\pi$ modulo $2\pi$, rather than the value given by $\operatorname{sgn}(K)\omega_\alpha(\kappa)$.
By Proposition~\ref{prop:apsidal}, the function $\omega_\alpha$ is a diffeomorphism from $(0,\kappa_c)$ onto $(\omega_c,\pi)$, where
\[
    \omega_c=\frac{2\pi}{\sqrt{2(\alpha+1)}}.
\]

\begin{lemma}[Geometry of the full-period image]\label{lem:full-period-image}
Let
\[
    \mathcal C_{++}:=\mathcal R(E_1\cap\{K\geq 0,\ w>0\}).
\]
For $\theta\in(\omega_c,\pi)$, let $\kappa(\theta)=\omega_\alpha^{-1}(\theta)$, extended continuously by
$\kappa(\omega_c)=\kappa_c$ and $\kappa(\pi)=0$. Then:

\begin{enumerate}
    \item The set $\mathcal C_{++}$ is the closed half-tongue
    \[
        \mathcal C_{++}
        =
        \left\{
            (1,\theta,z):
            \omega_c\leq \theta\leq \pi,
            \ z_{\mathrm{low}}(\kappa(\theta))\leq z\leq z_{\mathrm{up}}(\kappa(\theta))
        \right\},
    \]
    where
    \[
        z_{\mathrm{low}}(\kappa)=\frac{p_\alpha(\kappa)}{(\alpha+1)b(\kappa)^{\alpha+1}},
        \qquad
        z_{\mathrm{up}}(\kappa)=\frac{p_\alpha(\kappa)}{(\alpha+1)a(\kappa)^{\alpha+1}}.
    \]

    \item We call the slice at $\theta=\pi$ the antipodal spine
    \[
        \mathcal S
        =
        \left\{(1,\pi,z): z\geq \frac{\pi_\alpha}{\alpha+1}\right\}.
    \]
    Its lowest point is
    \[
        q_{\mathrm a}^+=\left(1,\pi,\frac{\pi_\alpha}{\alpha+1}\right).
    \]

    \item The boundary of $\mathcal C_{++}$ relative to $\{r=1\}$, away from the spine, is the image of the stratum $L=0$.
    Its two boundary branches meet at the circular endpoint
    \[
        q_{\mathrm c}^+
        =
        \left(
            1,
            \frac{2\pi}{\sqrt{2(\alpha+1)}},
            \frac{2\pi}{\sqrt{2\alpha(\alpha+1)}}
        \right),
    \]
    Its height is $z_c=2\pi/\sqrt{2\alpha(\alpha+1)}$, and the corresponding circular initial covector is
    \[
        \lambda_{\mathrm c}^+
        =
        \left(
            0,
            \sqrt{\frac{\alpha}{\alpha+1}},
            \frac{1}{\sqrt{\alpha+1}}
        \right).
    \]

    \item If
    \[
        \mathcal C_+:=\mathcal R(E_1\cap\{w>0\}),
    \]
    then $\mathcal C_+$ is obtained from $\mathcal C_{++}$ by reflection across the antipodal meridian $\theta=\pi$. The full image is
    \[
        \operatorname{Im}(\mathcal R)=\mathcal C_+\cup \mathcal C_-,
    \]
    where $\mathcal C_-$ is the reflection of $\mathcal C_+$ across the plane $z=0$.
\end{enumerate}
\end{lemma}

\begin{proof}
Restrict first to $w>0$ and $K\geq 0$. Then $s=w^{1/\alpha}$ and $\kappa=Ks$, so for fixed $\kappa\in[0,\kappa_c]$ the admissible values of $s$ are exactly the interval $[a(\kappa),b(\kappa)]$. By \eqref{eq:cut-param}, the angular coordinate depends only on $\kappa$, while the height is strictly decreasing in $s$. Thus the image at fixed $\kappa$ is exactly the vertical interval between $z_{\mathrm{low}}(\kappa)$ and $z_{\mathrm{up}}(\kappa)$. Since $\omega_\alpha$ is a diffeomorphism from $(0,\kappa_c)$ onto $(\omega_c,\pi)$, this proves the description of $\mathcal C_{++}$.

At $\kappa=0$ one has $a(0)=0$, $b(0)=1$, and $p_\alpha(0)=\pi_\alpha$, so the image is the antipodal spine claimed above. For $0<\kappa<\kappa_c$, the boundary values $s=a(\kappa)$ and $s=b(\kappa)$ are exactly the covectors with $L=0$, so the boundary of the half-tongue away from the spine is their image. When $\kappa=\kappa_c$, the two roots coalesce at $s_c=(\alpha+1)^{-1/(2\alpha)}$, giving the circular endpoint $q_{\mathrm c}^+$.

Finally, the symmetries $K\mapsto -K$ and $w\mapsto -w$ act on the endpoint by reflection across the antipodal meridian and across the plane $z=0$, respectively. This produces the full image $\operatorname{Im}(\mathcal R)=\mathcal C_+\cup\mathcal C_-$.
\end{proof}

We now set
\[
    \Cut^*(q_0):=\operatorname{Im}(\mathcal R)=\mathcal C_+\cup\mathcal C_-.
\]
This is the candidate cut locus for the Extended Hadamard argument.

\begin{figure}[htbp]
    \centering
    \includegraphics[width=0.85\linewidth]{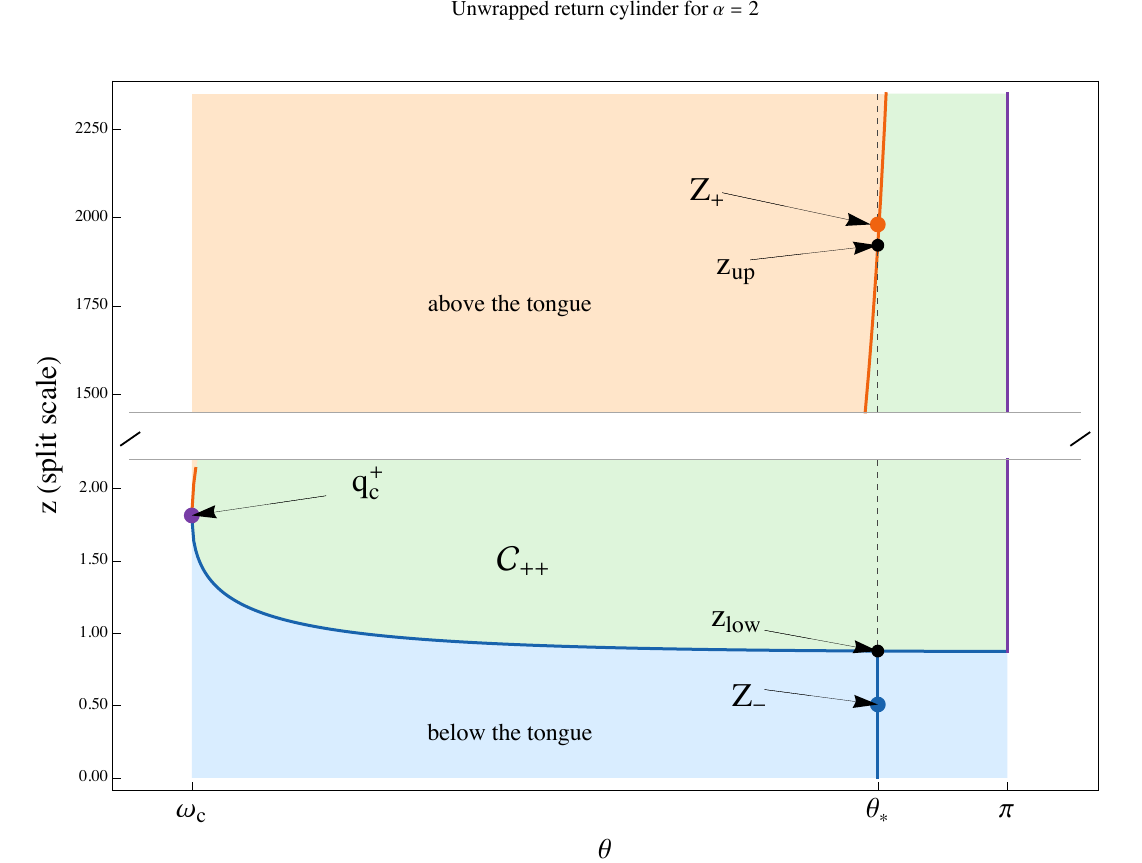}
    \caption{Unwrapped cylinder $\{r=1\}$ into the $(\theta,z)$-plane. Positive tongue $\mathcal C_{++}$ between $\theta=\omega_c$ and $\theta=\pi$, with lower boundary $z_{\mathrm{low}}$, upper boundary $z_{\mathrm{up}}$, the antipodal spine at $\theta=\pi$, and the circular endpoint at $\theta=\omega_c$. As $\kappa\uparrow\kappa(\theta)$ these curves approach the tongue boundary, while as $\kappa\downarrow0$ the lower family tends to $z=0$ and the upper family tends to $z=+\infty$. Here the vertical scale is split into two disjoint regimes.}
    \label{unwrapped cylinder and returns}
\end{figure}
The proof now reduces to four points. First, by Theorem~\ref{thm:conjugate-comparison}, the exponential map has no conjugate points on the domain $\mathcal N$, so assumption~\ref{loc diffeo EHT} of Theorem~\ref{extended Hadamard} is already verified. Second, we must show that no geodesic meets $\Cut^*(q_0)$ before its full radial period. Third, we must prove the properness required in assumption~\ref{proper EHT}. Finally, we note that $\mathbb R^3\setminus\Cut^*(q_0)$ is connected and simply connected, which gives assumption~\ref{simply conn EHT}. The key geometric step is the second one.

For a noncircular geodesic, we call a time $t\in(0,P_\alpha(\lambda))$ with $r(t;\lambda)=1$ a \emph{premature cylinder-return time}. The first positive cylinder-return time is the least $t>0$ with $r(t;\lambda)=1$, when it exists. We therefore analyze this first return to the cylinder $r=1$. Since $\Cut^*(q_0)\subset\{r=1\}$, this is exactly where a premature intersection with the candidate cut locus could occur.

Assume first that $K>0$, $w>0$, and that $(\kappa,s)$ is an interior parameter, so that $L=\pm\sqrt{Q_\kappa(s)}$. Let $\tau_+(\kappa,s)$, $\Theta_+(\kappa,s)$, and $Z_+(\kappa,s)$ denote the return time, angular coordinate, and height at the first return to $r=1$ for the branch $L>0$. Likewise, let $\tau_-(\kappa,s)$, $\Theta_-(\kappa,s)$, and $Z_-(\kappa,s)$ denote the corresponding quantities for the branch $L<0$. Explicitly,
\begin{align}
    \tau_+(\kappa,s)&=\frac{2}{s}\int_s^{b(\kappa)}\frac{\dd u}{\sqrt{Q_\kappa(u)}},
    &
    \Theta_+(\kappa,s)&=2\kappa\int_s^{b(\kappa)}\frac{\dd u}{u^2\sqrt{Q_\kappa(u)}},\label{eq:return-plus}\\
    \tau_-(\kappa,s)&=\frac{2}{s}\int_{a(\kappa)}^s\frac{\dd u}{\sqrt{Q_\kappa(u)}},
    &
    \Theta_-(\kappa,s)&=2\kappa\int_{a(\kappa)}^s\frac{\dd u}{u^2\sqrt{Q_\kappa(u)}}.\label{eq:return-minus}
\end{align}
The corresponding return heights may be written in the form
\begin{equation}\label{eq:return-heights}
    Z_\pm(\kappa,s)
    =
    \frac{\tau_\pm(\kappa,s)\pm 2\sqrt{Q_\kappa(s)}}{(\alpha+1)s^\alpha},
\end{equation}
which follows from the virial identity \eqref{eq:virial}.

The next proposition is the geometric heart of the section. It says that, at a fixed return angle such that an intersection with the conjectured cut locus is even possible, the inward first return $(L<0)$ endpoint map takes values strictly below the conjectured cut locus, and the outward first return endpoint map $(L>0)$ takes values strictly above it. Here ``below'' and ``above'' refer to the vertical coordinate at a fixed angular coordinate, compared with the lower and upper heights of the tongue, respectively. This is the same argument used in \cite{albert2025geodesicsgrushinspaces}. We will again reason by monotonicity. 

\begin{proposition}[Separation of first returns]\label{prop:first-return-separation}
Fix $\theta\in(\omega_c,\pi)$ and let $\kappa(\theta)\in(0,\kappa_c)$ be determined by $\omega_\alpha(\kappa(\theta))=\theta$.

\begin{enumerate}
    \item For every $0<\kappa<\kappa(\theta)$, there exist unique values $s_-(\kappa;\theta)$ and $s_+(\kappa;\theta)$ in $(a(\kappa),b(\kappa))$ such that the corresponding first returns satisfy
    \[
        \Theta_-(\kappa,s_-(\kappa;\theta))=\theta,
        \qquad
        \Theta_+(\kappa,s_+(\kappa;\theta))=\theta.
    \]

    \item Define the corresponding return heights by
    \[
        Z_-(\kappa;\theta):=Z_-(\kappa,s_-(\kappa;\theta)),
        \qquad
        Z_+(\kappa;\theta):=Z_+(\kappa,s_+(\kappa;\theta)).
    \]
    Then $Z_-(\cdot;\theta)$ is strictly increasing and $Z_+(\cdot;\theta)$ is strictly decreasing on $(0,\kappa(\theta))$. Moreover,
    \[
        \lim_{\kappa\downarrow0} Z_-(\kappa;\theta)=0,
        \qquad
        \lim_{\kappa\uparrow\kappa(\theta)} Z_-(\kappa;\theta)=z_{\mathrm{low}}(\kappa(\theta)),
    \]
    and
    \[
        \lim_{\kappa\downarrow0} Z_+(\kappa;\theta)=+\infty,
        \qquad
        \lim_{\kappa\uparrow\kappa(\theta)} Z_+(\kappa;\theta)=z_{\mathrm{up}}(\kappa(\theta)).
    \]
    In particular,
    \[
        0<Z_-(\kappa;\theta)<z_{\mathrm{low}}(\kappa(\theta))<z_{\mathrm{up}}(\kappa(\theta))<Z_+(\kappa;\theta)
    \]
    for every $0<\kappa<\kappa(\theta)$.

    \item At the limiting angle $\theta=\omega_c$, with $\kappa(\theta)=\kappa_c$, both height functions retain the asserted strict monotonicity and endpoint limits, and $Z_-(\kappa;\omega_c)<z_c<Z_+(\kappa;\omega_c)$ for $0<\kappa<\kappa_c$, while
    \[
        z_{\mathrm{low}}(\kappa_c)=z_{\mathrm{up}}(\kappa_c)=z_c.
    \]

    \item Suppose $K=0$ and $w>0$. Write the geodesic in signed radial coordinates $(\rho,z)$ in its vertical plane, with $\rho(0)=1$ and $\dot\rho(0)=L$. Then:
    \begin{enumerate}
        \item if $L>0$, there is a unique premature return to $r=1$, and it lies on the original meridian $\theta=0$;
        \item if $L=0$, there is no return to $r=1$ before time $P_\alpha$, and at time $P_\alpha$ the endpoint is
        \[
            \left(1,\pi,\frac{\pi_\alpha}{\alpha+1}\right);
        \]
        \item if $L<0$, there is a unique premature return to $r=1$, it lies on the antipodal meridian $\theta=\pi$, and its height satisfies
        \[
            0<z<\frac{\pi_\alpha}{\alpha+1}.
        \]
    \end{enumerate}
\end{enumerate}
\end{proposition}

\begin{proof}
For fixed $\kappa\in(0,\kappa_c)$, the angular advances satisfy
\[
    \Theta_+(\kappa,s)+\Theta_-(\kappa,s)=\omega_\alpha(\kappa),
\]
and differentiating in $s$, we have
\[
    \partial_s\Theta_+(\kappa,s)<0,
    \qquad
    \partial_s\Theta_-(\kappa,s)>0.
\]
By continuous extension to the endpoints, $\Theta_-(\kappa,a(\kappa))=\Theta_+(\kappa,b(\kappa))=0$ and $\Theta_-(\kappa,b(\kappa))=\Theta_+(\kappa,a(\kappa))=\omega_\alpha(\kappa)$. Since $\omega_\alpha$ is strictly decreasing, the condition $\omega_\alpha(\kappa)>\theta$ is equivalent to $0<\kappa<\kappa(\theta)$. The intermediate value theorem therefore gives the unique values $s_\pm(\kappa;\theta)$, and the implicit function theorem gives their smooth dependence on $\kappa$.

We next compare heights. Let $G_\pm=(\Theta_\pm,Z_\pm)$ be the first-return endpoint maps. Differentiating the identity $r(\tau_\pm(\kappa,s),\kappa,s)=1$ gives
\[
    \det DG_\pm(\kappa,s)
    =
    \frac{1}{\dot r(\tau_\pm(\kappa,s))}
    \det\left.\frac{\partial(r,\theta,z)}{\partial(t,\kappa,s)}\right|_{t=\tau_\pm(\kappa,s)}.
\]
For interior parameters one has $0<\tau_\pm<P_\alpha$ and $\dot r(\tau_\pm)=-\dot r(0)\neq 0$. By Theorem~\ref{thm:conjugate-comparison}, the endpoint Jacobian is nonsingular before time $P_\alpha$, so $\det DG_\pm\neq 0$. Along the level curve $\Theta_\pm=\theta$, this implies that the derivative of $Z_\pm(\kappa;\theta)$ with respect to $\kappa$ never vanishes. Hence each \emph{height curve} $Z_{\pm}(\cdot;\theta)$ is strictly monotone.

To determine the sign of this monotonicity, it suffices to compute the endpoint limits. As $\kappa\uparrow\kappa(\theta)$, the relation $\Theta_++\Theta_-=\omega_\alpha(\kappa)$ shows that $s_+(\kappa;\theta)\to a(\kappa(\theta))$ and $s_-(\kappa;\theta)\to b(\kappa(\theta))$. The corresponding first-return arcs converge to the full radial cycles based at the inner and outer turning radii, so
\[
    Z_+(\kappa;\theta)\to z_{\mathrm{up}}(\kappa(\theta)),
    \qquad
    Z_-(\kappa;\theta)\to z_{\mathrm{low}}(\kappa(\theta))
\]
as $\kappa\uparrow\kappa(\theta)$.

As $\kappa\downarrow0$, the inward family approaches the horizontal Euclidean chord joining $q_0$ to $(\cos\theta,\sin\theta,0)$. More precisely, after the rescaling $u=s_-v$ in the integrals defining $\tau_-$ and $\Theta_-$, a standard dominated-convergence argument shows that $\tau_-$ remains bounded while $w=s_-^\alpha\to0$. Hence $Z_-(\kappa;\theta)\to0$.

For the outward family, one first shows that $s_+(\kappa;\theta)\to0$. If not, then the integral defining $\Theta_+$ would converge to $0$, contradicting the identity $\Theta_+=\theta>0$. If $Z_+(\kappa;\theta)$ remained bounded, the virial identity would force both $\tau_+$ and $L$ to tend to $0$. The unit speed identity would then imply $K\to1$, and the return arc would remain arbitrarily close to the circle $r=1$, forcing again $\Theta_+\to0$, a contradiction. Thus $Z_+(\kappa;\theta)\to+\infty$.

Since $Z_-(\cdot;\theta)$ and $Z_+(\cdot;\theta)$ are strictly monotone and approach the indicated limits at the two ends of the interval $(0,\kappa(\theta))$, the ordering
\[
    0<Z_-(\kappa;\theta)<z_{\mathrm{low}}(\kappa(\theta))<z_{\mathrm{up}}(\kappa(\theta))<Z_+(\kappa;\theta)
\]
follows immediately.

When $\theta=\omega_c$, the same argument applies, except that the two tongue boundaries coalesce at the circular geodesic's endpoint value. This gives the statement in part~(3).

Finally, consider the stratum $K=0$, still with $w>0$. In signed radial coordinates the motion satisfies
\[
    \dot\rho^2+w^2|\rho|^{2\alpha}=1.
\]
If $L>0$, the trajectory moves first to the positive turning point and then returns to $\rho=1$, so its premature return lies on the original meridian. If $L=0$, then $w=1$ and the first positive return to $r=1$ occurs only at time $P_\alpha$, when the endpoint is exactly the lowest point of the antipodal spine. If $L<0$, then $0<w<1$, and the trajectory crosses $\rho=-1$ before it reaches the negative turning point. This gives the unique premature return to the antipodal meridian, and the integral formula for its height shows that it lies strictly below the spine.
\end{proof}
The restriction $\theta\geq\omega_c$ in
Proposition~\ref{prop:first-return-separation} is imposed because
the candidate cut locus has no points with $|\theta|<\omega_c$.
The first-return families themselves exist also for
$0<\theta<\omega_c$.

Indeed, since $\omega_\alpha(\kappa)>\omega_c>\theta$ for every
$0<\kappa<\kappa_c$, there are unique first returns of both radial
phases with angular coordinate $\theta$.  The nonsingularity argument for the first-return endpoint Jacobian used above
shows that their heights are strictly monotone.  As
$\kappa\downarrow0$ their heights tend respectively to $0$ and
$+\infty$, while as $\kappa\uparrow\kappa_c$ both families converge
to the circular geodesic and hence to $\left(1,\theta,\frac{\theta}{\sqrt\alpha}\right)$. 
\begin{corollary}\label{cor:hadamard-image-inclusion}
It holds that
\[
    \Exp_{q_0}(\mathcal N)\subset \mathbb R^3\setminus \Cut^*(q_0).
\]
\end{corollary}

\begin{proof}
Since $\Cut^*(q_0)\subset \{r=1\}$, it is enough to show that no geodesic returns to the cylinder $r=1$ at a point of $\Cut^*(q_0)$ before time $P_\alpha$.

If $L\neq0$, this is exactly the content of Proposition~\ref{prop:first-return-separation}: for $Kw>0$, every premature inward return lies strictly below the tongue $\mathcal C_{++}$ and every premature outward return lies strictly above it. The symmetries in $K$ and $w$ give the same conclusion in the other quadrants.

If $L=0$ and $Kw\neq0$, then $r(0)=1$ is one of the two turning radii. The geodesic leaves the cylinder and returns to it for the first time only at $t=P_\alpha$, so there is no premature return at all. For the circular orbit, one has $r(t)\equiv1$ and $0<\theta(t)<\omega_c$ for $0<t<P_\alpha$, whereas the tongue begins at $\theta=\omega_c$. Finally, if $w=0$, the geodesic is a horizontal Euclidean line lying in the plane $z=0$, which is disjoint from $\Cut^*(q_0)$.

This proves the inclusion.
\end{proof}

We next observe that the complement of the candidate cut locus has the topology required in assumption~\ref{simply conn EHT}.

\begin{lemma}\label{lem:cut-complement-topology}
The set $\mathbb R^3\setminus\Cut^*(q_0)$ is connected and simply connected.
\end{lemma}

\begin{proof}
Both components $\mathcal C_+$ and $\mathcal C_-$ lie in the cylinder $\{r=1\}$ and avoid the complete meridian $\{\theta=0\}$. Cutting the cylinder along this meridian and unwrapping it identifies each tongue with a planar region whose boundary is a properly embedded curve, by the explicit graph description in Lemma~\ref{lem:full-period-image}. Here a properly embedded half-plane means the image of a proper topological embedding $[0,\infty)\times\mathbb R\hookrightarrow\mathbb R^3$. The two regions lie in the disjoint half-spaces $z>0$ and $z<0$, so their boundary curves can be straightened compatibly in the unwrapped cylinder; extending this straightening through a collar of the cylinder gives an ambient homeomorphism taking the pair to two disjoint standard half-planes. The complement of the latter is connected and simply connected, and therefore so is $\mathbb R^3\setminus\Cut^*(q_0)$.
\end{proof}

The only remaining input for the Extended Hadamard theorem is properness.

\begin{lemma}[Properness]\label{lem:proper}
The restriction
\[
    \Exp_{q_0}|_{\mathcal N}:\mathcal N\to \mathbb R^3\setminus \Cut^*(q_0)
\]
is proper.
\end{lemma}

\begin{proof}
Let $C\subset \mathbb R^3\setminus\Cut^*(q_0)$ be compact, and let
\[
    t_n\lambda_n\in \bigl(\Exp_{q_0}|_{\mathcal N}\bigr)^{-1}(C),
    \qquad
    \lambda_n=(L_n,K_n,w_n)\in E_1.
\]
Since the unit shell $E_1$ is compact, after passing to a subsequence we may assume that $\lambda_n\to\lambda=(L,K,w)\in E_1$.

We first show that $(t_n)$ is bounded. Because $\Exp_{q_0}(t_n\lambda_n)$ remains in the compact set $C$, the corresponding endpoint coordinates $r_n$ and $z_n$ are bounded. On the unit shell one has $|L_n|\leq1$, $|w_n|\leq1$, and the radial energy identity gives $|\dot r_n|\leq1$. The virial identity \eqref{eq:virial} gives $t_n=(\alpha+1)w_nz_n+r_n\dot r_n-L_n$, so $t_n$ is bounded.

Passing to a further subsequence, we may assume that $t_n\to t\geq0$. We claim that $t<P_\alpha(\lambda)$. If $w=0$, then $P_\alpha(\lambda)=+\infty$ and there is nothing to prove. Assume therefore that $w\neq0$. Since $t_n<P_\alpha(\lambda_n)$ and the period function is continuous on $\{w\neq0\}$, we have $t\leq P_\alpha(\lambda)$. If equality held, then by continuity of the exponential map,
\[
    \Exp_{q_0}(t_n\lambda_n)
    \longrightarrow
    \Exp_{q_0}(P_\alpha(\lambda)\lambda)
    =
    \mathcal R(\lambda)
    \in \Cut^*(q_0).
\]
But the left-hand side converges inside the closed set $C$, contradicting $C\cap\Cut^*(q_0)=\varnothing$. Hence $t<P_\alpha(\lambda)$, so $t\lambda\in\mathcal N$.

Finally, continuity of the exponential map gives
\[
    \Exp_{q_0}(t\lambda)
    =
    \lim_{n\to\infty}\Exp_{q_0}(t_n\lambda_n)
    \in C.
\]
Thus every sequence in the preimage of $C$ has a convergent subsequence whose limit remains in the same preimage. This proves that $\Exp_{q_0}|_{\mathcal N}$ is proper.
\end{proof}

We may now conclude the Extended Hadamard argument. The upper bound $t_{\operatorname{cut}}\le P_\alpha$ has been established above for every geodesic. By Theorem~\ref{thm:conjugate-comparison}, the exponential map has invertible differential throughout $\mathcal N$, verifying assumption~\ref{loc diffeo EHT}. Corollary~\ref{cor:hadamard-image-inclusion} gives inclusion in the complement, and Lemma~\ref{lem:proper} gives properness. The image of this proper local diffeomorphism is nonempty, open, and closed in the complement; by connectedness from Lemma~\ref{lem:cut-complement-topology}, it is the whole complement, verifying assumption~\ref{inclusion EHT}. Lemma~\ref{lem:cut-complement-topology} also gives assumption~\ref{simply conn EHT}. Theorem~\ref{extended Hadamard} therefore implies that
\[
    t_{\mathrm{cut}}=P_\alpha,
    \qquad
    \Cut(q_0)=\Cut^*(q_0)=\mathcal C_+\cup \mathcal C_-.
\]
for the normalized base point $q_0=(1,0,0)$.

We summarize the resulting synthesis for an arbitrary Riemannian base point.

\begin{theorem}[Optimal synthesis from a Riemannian point]\label{thm:riemannian-synthesis}
Let $\alpha>1$ and let
\[
    q_0=(r_0\cos\theta_0,r_0\sin\theta_0,z_0),
    \qquad r_0>0.
\]
Let $\gamma$ be a unit-speed geodesic beginning at $q_0$, with initial covector $\lambda$ having momentum coordinates $(L,K,w)$. Then:

\begin{enumerate}
    \item If $w=0$, then
    \[
        t_{\mathrm{cut}}(\gamma)=t_{\mathrm{con}}(\gamma)=+\infty.
    \]

    \item If $w\neq0$, then
    \[
        t_{\mathrm{cut}}(\gamma)=P_\alpha(K,w),
    \]
    where, for a noncircular orbit,
    \begin{equation}\label{eq:final-cut-period}
        P_\alpha(K,w)=2\int_{r_-}^{r_+}\frac{\dd r}{\sqrt{1-K^2/r^2-w^2r^{2\alpha}}}.
    \end{equation}
    Here $r_-<r_+$ are the radial turning points; when $K=0$, we set $r_-=0$ and $r_+=|w|^{-1/\alpha}$. For the circular orbit, \eqref{eq:final-cut-period} is understood by continuous extension and
    \[
        P_\alpha=\frac{2\pi r_0}{\sqrt{2\alpha}}.
    \]

    \item If $w\neq0$, then
    \[
        t_{\mathrm{con}}(\gamma)=P_\alpha(K,w)
        \quad\Longleftrightarrow\quad
        L=0,
    \]
    while
    \[
        t_{\mathrm{con}}(\gamma)>P_\alpha(K,w)
        \qquad\text{if }L\neq0.
    \]

    \item The cut locus is contained in the cylinder $\{r=r_0\}$. Put
    \[
        \omega_c=\frac{2\pi}{\sqrt{2(\alpha+1)}}
    \]
    and define the positive-$K$, positive-$z$ quarter of the cut locus by
    \[
        \mathcal C_{++}(q_0)
        =
        \left\{
            \left(
                r_0,
                \theta_0+\omega_\alpha(\kappa),
                z_0+\frac{r_0^{\alpha+1}p_\alpha(\kappa)}{(\alpha+1)s^{\alpha+1}}
            \right):
            0\leq\kappa\leq\kappa_c,
            \ a(\kappa)\leq s\leq b(\kappa),
            \ s>0
        \right\}.
    \]
    Let $\mathcal C_+(q_0)$ be the union of $\mathcal C_{++}(q_0)$ with its reflection across the antipodal meridian $\theta=\theta_0+\pi$, and let $\mathcal C_-(q_0)$ be the reflection of $\mathcal C_+(q_0)$ across the plane $z=z_0$. Then
    \[
        \Cut(q_0)=\mathcal C_+(q_0)\cup \mathcal C_-(q_0).
    \]
\end{enumerate}

The sets $\mathcal C_\pm(q_0)$ are the two tongue-shaped components of the cut locus. Their antipodal spines are
\[
    \left\{(r_0,\theta_0+\pi,z): z\geq z_0+\frac{\pi_\alpha}{\alpha+1}r_0^{\alpha+1}\right\}
\]
and
\[
    \left\{(r_0,\theta_0+\pi,z): z\leq z_0-\frac{\pi_\alpha}{\alpha+1}r_0^{\alpha+1}\right\}.
\]
In particular, the points nearest to $z=z_0$ on the antipodal meridian are
\[
    q_{\mathrm a}^\pm
    =
    \left(
        r_0,
        \theta_0+\pi,
        z_0\pm\frac{\pi_\alpha}{\alpha+1}r_0^{\alpha+1}
    \right).
\]

The boundary of each tongue relative to the cylinder $\{r=r_0\}$ is the image of the stratum $L=0$. For the positive tongue, the two boundary branches meet at the circular endpoints
\[
    q_{\mathrm c,\pm}^+
    =
    \left(
        r_0,
        \theta_0\pm\omega_c,
        z_0+\frac{2\pi r_0^{\alpha+1}}{\sqrt{2\alpha(\alpha+1)}}
    \right),
\]
and the corresponding endpoints of the negative tongue are their reflections across $z=z_0$. We say that a cut point is cut-conjugate to $q_0$ if some minimizing geodesic reaching it has equal cut and conjugate times. Consequently, the cut-conjugate points are precisely the points on the relative boundaries of the two tongues.
\end{theorem}

\begin{proof}
The anisotropic dilation
$\delta_{1/r_0}(x,y,z)=\left(\frac{x}{r_0},\frac{y}{r_0},\frac{z}{r_0^{\alpha+1}}\right)$
is a homothety of factor $r_0^{-1}$ and sends $q_0$ to $(1,0,0)$. Under this normalization, the momentum coordinates transform by
$(L,K,w)\longmapsto \left(\frac{L}{r_0},\frac{K}{r_0},r_0^\alpha w\right)$,
so that $\kappa$ is invariant and the normalized parameter is $s=r_0|w|^{1/\alpha}$. The theorem therefore follows by transporting the normalized synthesis back under these symmetries.
\end{proof}
\begin{figure}[!htbp]
    \centering
    \includegraphics[scale=0.6]{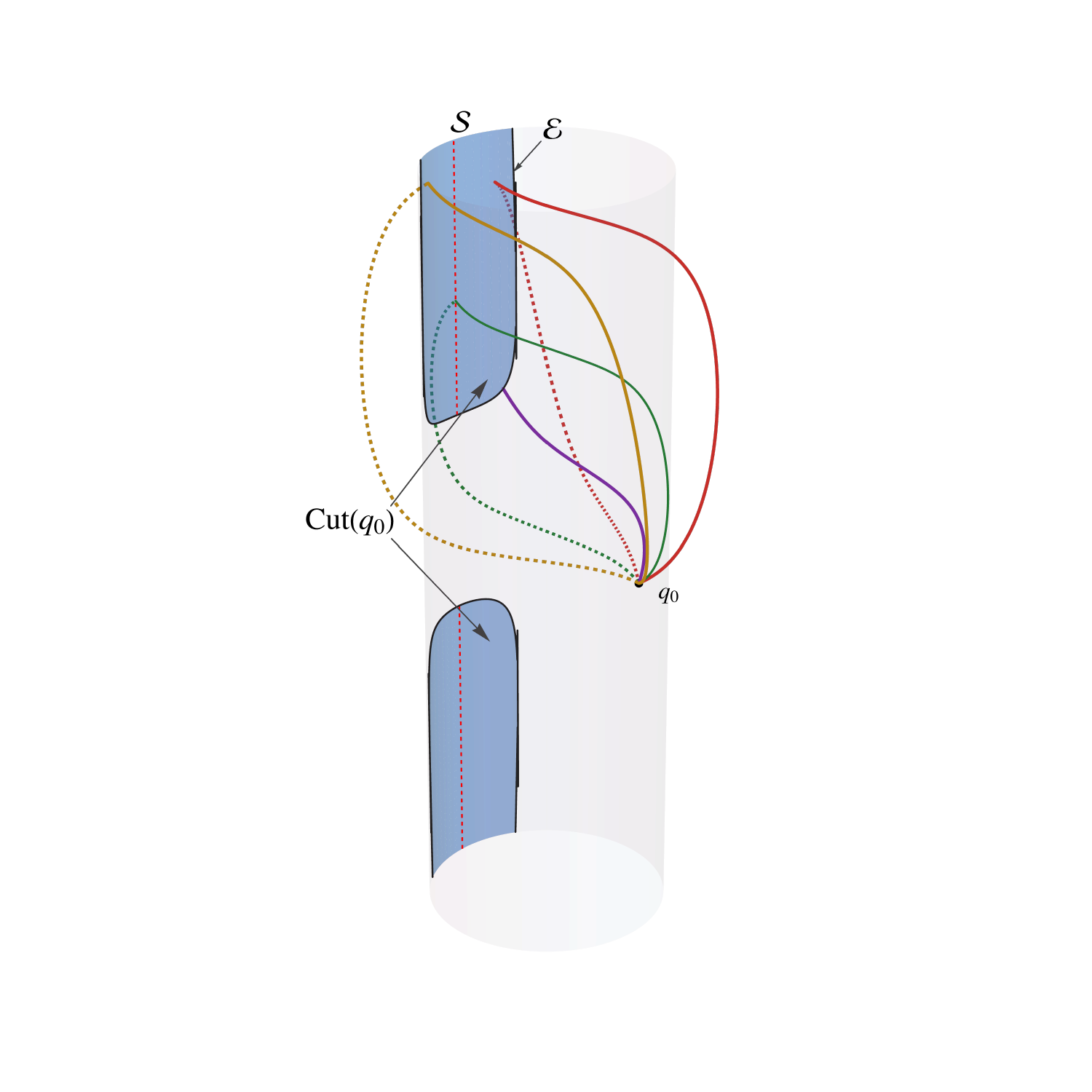}
    \caption{Cut locus at a Riemannian point $q_0$. The antipodal spine
$\mathcal{S}$ corresponds to $K=0$ geodesics, while $\mathcal{E}$
denotes the union of the tongue boundaries relative to the cylinder,
consisting of cut-conjugate points corresponding to $L=0$. The dashed
geodesics indicate the radial-phase partners obtained from the Maxwell
symmetry $(L,K,w)\mapsto(-L,K,w)$ described in
Theorem~\ref{Geodesic and Cut time Theorem}. The green geodesics have
angular momentum $K=0$, remain confined to the vertical plane through
$q_0$ and the singular axis, and reach the antipodal spine
$\mathcal{S}$. The red and gold geodesics have initial covectors in
the generic regime. The purple geodesic has initial radial momentum
$L=0$, is its own radial-phase partner, and reaches the cut locus at
a point conjugate to $q_0$ lying in $\mathcal{E}$.}
    \label{cut locus fig}
\end{figure}
\clearpage
\section{Beyond the Homogeneous Family}
We examine when the radial-period Maxwell time $\mathbf P$ is sharp beyond the homogeneous family. Using circular geodesics, we derive necessary conditions for radial-period optimality, construct a profile violating them, and present numerical evidence that these conditions are not sufficient. 
\begin{theorem}[Necessary conditions for radial-period optimality]
\label{thm:necessary-radial-period}
Let \(f\in\mathfrak{F}^{\mathbf{P}}_R\). Suppose that $t_{\operatorname{cut}}=\mathbf{P}$ for every Riemannian base point \(q_0\notin\Sigma\) and every normal geodesic
with $w\neq 0$. Then
\[
    D_f(r):=\frac{r\Psi'(r)}{\Psi(r)}\geq 4,
    \qquad
    \Psi(r):=r^3f(r)f'(r),
\]
for every \(r>0\). In particular,
$\eta(r)\geq 1$.
\end{theorem}

\begin{proof}
Fix \(r_0>0\), and take \(q_0=(r_0,0,0)\). Consider a unit-speed
radially constant geodesic with momentum coordinates \((0,K,w)\), where
\(Kw\neq0\). The circularity condition is $K^2=w^2\Psi(r_0)$. The corresponding radial period is
\[
    \mathbf{P}
    =
    \frac{2\pi r_0^{3/2}}
    {|w|\sqrt{\Psi'(r_0)}},
\]
with \(\mathbf{P}=+\infty\) if \(\Psi'(r_0)=0\). Replacing \(K\) by \(-K\) gives a second geodesic with the same radial and
vertical components. Since their angular velocities are
$\dot\theta=\pm\frac{|K|}{r_0^2}$
they meet after the time
\[
    T_{\mathrm{ang}}
    =
    \frac{\pi r_0^2}{|K|}.
\]
Hence $t_{\mathrm{cut}}\leq T_{\mathrm{ang}}$. By hypothesis \(t_{\mathrm{cut}}=\mathbf P\), so \(\Psi'(r_0)>0\) and
\[
    \frac{2\pi r_0^{3/2}}
    {|w|\sqrt{\Psi'(r_0)}}
    \leq
    \frac{\pi r_0^2}{|K|}.
\]
Using \(K^2=w^2\Psi(r_0)\), we obtain
\[
    \frac{r_0\Psi'(r_0)}{\Psi(r_0)}\geq4.
\]
Since \(r_0>0\) was arbitrary, \(D_f\geq4\) on \((0,\infty)\). Finally, $\Psi(r)=r^2\eta(r)f(r)^2$ and therefore
\[
    D_f(r)
    =
    2+2\eta(r)
    +\frac{r\eta'(r)}{\eta(r)}.
\]
By \ref{A4}, \(\eta'\leq0\), so
$4\leq D_f(r)\leq2+2\eta(r)$,
which gives \(\eta(r)\geq1\).
\end{proof}
One may readily check that the power laws $r\mapsto r^\alpha$ satisfy $D_f(r)=2+2\eta(r)=2+2\alpha>4$.
\begin{corollary}
\label{cor:radial-period-not-always-cut}
There exists \(f\in\mathfrak F_R^{\mathbf{P}}\) and a normal geodesic for the corresponding CC-distance for which $\mathbf{P}$ is not its cut time.
In fact, this holds for
\[
    f(r)=r^2(1+r^{20})^{-1/20}.
\]
\end{corollary}

\begin{proof}
We have $\eta(r)=1+(1+r^{20})^{-1}$,
which is strictly decreasing and greater than \(1\). Moreover,
$A(r)=r^2(2+r^{20})(1+r^{20})^{-11/10}$
and $rA'(r)\to 0$ as $r\downarrow 0$. Thus \(f\in\mathfrak F_R\). Writing
\[
    \Psi(r)=r^3f(r)f'(r)
    =r^6(2+r^{20})(1+r^{20})^{-11/10},
\]
a direct computation gives
$D_f(r)
    =\frac{r\Psi'(r)}{\Psi(r)}
    =
    \frac{4r^{40}-6r^{20}+12}
         {(1+r^{20})(2+r^{20})}$.
The numerator is strictly positive, so \(\Psi'>0\) and $f\in \mathfrak{F}_R^{\mathbf{P}}$. On the other hand,
\[
    D_f(1)=\frac53<4.
\]
The conclusion now follows from
Theorem~\ref{thm:necessary-radial-period}.
\end{proof}
Determining the largest class of functions in $\mathfrak{F}_R^{\mathbf{P}}$ for which $\mathbf P$ is the true cut time remains open. The following numerical example suggests that the necessary conditions of Theorem~\ref{thm:necessary-radial-period} are not sufficient.

One may check that the function $f(r)=re^{-1/r^2}$ with $f(0)=0$ belongs to $\mathfrak{F}_R^{\mathbf{P}}$ and satisfies the necessary conditions $D_f\geq4$ and $\eta\geq1$ of Theorem~\ref{thm:necessary-radial-period}. However, numerical integration
already exhibits premature conjugacy. At \(q_0=(1,0,0)\), consider the unit covector determined by
\[
    K=0.9,\qquad
    L\approx0.261534,\qquad
    w\approx0.947897.
\]
The radial period, computed both from the radial quadrature and from Hamiltonian integration, is
\[
    \mathbf{P}\approx3.702642.
\]
Integrating simultaneously the Hamiltonian system and its variational
equations, and locating the first positive zero of
\[
    \det
    \frac{\partial(x(t),y(t),z(t))}
         {\partial(u_0,v_0,w_0)},
\]
by finding successive positive and negative values within $10^{-3}$ of the following approximate value, we obtain
\[
    t_{\rm con}\approx3.695907
\]
The zero is transverse and the values are stable under refinement of the
integration tolerance.  Thus the numerical evidence indicates that
higher-order conditions on the profile are required to control conjugacy
before the radial return. It is interesting that in this example, the detected premature conjugacy is extremely close to $\mathbf{P}$. \footnote{We provide the \href{https://github.com/StevieMike/Code-for-Optimal-Synthesis-on-Radial-Grushin-Spaces/tree/main}{link to a GitHub repository} containing a Mathematica notebook which can reproduce this computation.} 

We suggest now $f\in \mathfrak{F}$ examples for further study. Corollary~\ref{cor:radial-period-not-always-cut} is covered by (1), while the numerical example $f(r)=re^{-1/r^2}$ just discussed is covered by (3). 
\begin{enumerate}
    \item For $a,b,c>0$ with $a+b\geq1$,
    \[
        f(r)
        =
        \frac{r^{a+b}}{(1+r^c)^{b/c}},\qquad \eta(r)
        =
        a+\frac{b}{1+r^c}.
    \]
    Here we have interpolation between the power laws $r^{a+b}$ near the
    singular axis and $r^a$ at infinity.

    \item For $a\geq0$, $b,c>0$, and $a+b\geq1$,
    \[
        f(r)
        =
        r^a\bigl(\log(1+r^c)\bigr)^{b/c},\qquad \eta(r)
        =
        a+
        b\frac{r^c}
        {(1+r^c)\log(1+r^c)}.
    \]
    When $a=0$ this gives examples in $\mathfrak F$
    which grow more slowly than every positive power of $r$ at
    infinity.

    \item For $a,c>0$,
    \[
        f(r)=r^a e^{-1/r^c},\qquad \eta(r)=a+\frac{c}{r^c}.
    \]
    These functions vanish to infinite order on the singular axis,
    while satisfying $f(r)\sim r^a$ as $r\to+\infty$.

    \item For $b,c,d>0$,
    \[
        f(r)
        =
        e^{-1/r^c}
        \bigl(\log(1+r^d)\bigr)^{b/d},
\qquad \eta(r)
        =
        \frac{c}{r^c}
        +
        b\frac{r^d}
        {(1+r^d)\log(1+r^d)}.
    \]
    Thus $\mathfrak F$ contains functions which vanish to infinite
    order at the singular axis and nevertheless have only
    logarithmic growth at infinity.

    \item For $a,b,c>0$ with $a+b\geq1$,
    \[
        f(r)
        =
        r^a\bigl(1-e^{-r^c}\bigr)^{b/c},\qquad \eta(r)
        =
        a+b\frac{r^c}{e^{r^c}-1}.
    \]
    This gives an exponentially fast transition from
    $f(r)\sim r^{a+b}$ near $r=0$ to $f(r)\sim r^a$ as
    $r\to+\infty$.
\end{enumerate}
With the parameter restrictions stated in (1)--(5), all five profiles belong to
$\mathfrak F_R$, and hence to $\mathfrak F$. The remaining conditions can be stated
exactly using the following one-variable quantities. For $t>0$, put
\[
 h_{\log}(t)=\frac{t}{(1+t)\log(1+t)},\qquad
 h_{\exp}(t)=\frac{t}{e^t-1},\qquad
 j_h(t)=-t h'(t)>0.
\]
For $k\in\{0,4\}$, define
\[
 \begin{aligned}
 R_k(a,b)&=\frac{(2+2a-k)(2a+b)+2ab
 +2\sqrt{a(a+b)(2+2a-k)(2+2a+2b-k)}}{b},\\
 T_k^h(a,b)&=\inf_{t>0}
 \frac{(2+2a+2b h(t)-k)(a+b h(t))}{b j_h(t)},
 \qquad h\in\{h_{\log},h_{\exp}\}.
 \end{aligned}
\]
The quantities with subscript $4$ are used only when $a\geq1$. Set also
\[
 U_0(a)=2(\sqrt{a+1}+\sqrt a)^2,\qquad
 U_4(a)=2(\sqrt a+\sqrt{a-1})^2\quad(a\geq1).
\]
For family~(4), write
\[
 \begin{aligned}
 H(t)&=c\,t^{-c/d}+b h_{\log}(t),\\
 \Delta_{b,c,d}(t)&=(2+2H(t))H(t)-c^2t^{-c/d}-bd\,j_{h_{\log}}(t).
 \end{aligned}
\]
The following table gives necessary and sufficient conditions within the parameter
ranges specified in the five examples. 

\begin{table}[ht]
\centering
\small
\renewcommand{\arraystretch}{1.3}
\begin{tabular}{c c l l l}
\hline
Family & $\mathfrak F,\mathfrak F_R$ & $\mathfrak F_R^{\mathbf{P}}$ precisely when
 & $D_f\geq4$ precisely when & $\eta\geq1$ precisely when\\
\hline
(1) & Yes & $c\leq R_0(a,b)$ & $a\geq1,\ c\leq R_4(a,b)$ & $a\geq1$\\
(2) & Yes & $c\leq T_0^{h_{\log}}(a,b)$ & $a\geq1,\ c\leq T_4^{h_{\log}}(a,b)$ & $a\geq1$\\
(3) & Yes & $c\leq U_0(a)$ & $a\geq1,\ c\leq U_4(a)$ & $a\geq1$\\
(4) & Yes & $\Delta_{b,c,d}(t)\geq0$ for all $t>0$ & Never & Never\\
(5) & Yes & $c\leq T_0^{h_{\exp}}(a,b)$ & $a>1,\ c\leq T_4^{h_{\exp}}(a,b)$ & $a\geq1$\\
\hline
\end{tabular}
\caption{Membership in $\mathfrak{F},\mathfrak{F}_R$, $\mathfrak{F}_R^{\mathbf{P}}$ and the necessary inequalities from Theorem~\ref{thm:necessary-radial-period} based on parameters. Conditions include equality where indicated.}
\label{tab:five-profiles}
\end{table}

Indeed, $\Psi=r^2f^2\eta$ gives
$D_f=2+2\eta+r\eta'/\eta$. Each profile is real analytic on $(0,\infty)$,
so $\Psi$ is strictly increasing exactly when $D_f\geq0$ everywhere: isolated
zeros of $\Psi'$ do not affect strict monotonicity. Substituting the displayed
elasticities yields the table; the bounds for (1) and (3) follow by minimizing
quadratic expressions, while the other rows reduce to the stated one-variable
tests. For instance, family~(4) belongs to $\mathfrak F_R^{\mathbf P}$ whenever
$0<c\leq2$ and $0<d\leq2$, for any $b>0$. 
\newpage
\noindent\textbf{AI Declaration:} In accordance with the \textit{Leiden Declaration on Artificial Intelligence and Mathematics}, we disclose that ChatGPT 5.6 was used for developing the overall proof architecture of Section~5, to streamline the exposition, and to generate the Mathematica code for the figures and the numerical computations in Section~7. The author accepts full responsibility for the correctness of the mathematical arguments and computational results.
\bibliographystyle{plain}
\bibliography{Radial_Grushin}
\end{document}